\documentclass[11pt]{article}     

\usepackage{graphicx}
\usepackage{url}
\usepackage{color}

\usepackage[mathscr]{euscript}		

\usepackage{dsfont}

\usepackage{psfrag}			
\usepackage{hyperref}

\usepackage{amsmath}
\usepackage{amsthm}

\usepackage{amssymb}

\usepackage{caption} 

\usepackage{anysize}

\usepackage{wasysym}

\usepackage[shortlabels]{enumitem}

\newtheorem{theorem}{Theorem}[section]
\newtheorem{lemma}[theorem]{Lemma}

\newtheorem{proposition}[theorem]{Proposition}
\newtheorem{remark}[theorem]{Remark}

\newtheorem{definition}[theorem]{Definition}

\numberwithin{equation}{section}

\definecolor{Red}{rgb}{1,0,0}
\definecolor{Blue}{rgb}{0,0,1}
\definecolor{Olive}{rgb}{0.41,0.55,0.13}
\definecolor{Yarok}{rgb}{0,0.5,0}
\definecolor{Green}{rgb}{0,1,0}
\definecolor{MGreen}{rgb}{0,0.8,0}
\definecolor{DGreen}{rgb}{0,0.55,0}
\definecolor{Yellow}{rgb}{1,1,0}
\definecolor{Cyan}{rgb}{0,1,1}
\definecolor{Magenta}{rgb}{1,0,1}
\definecolor{Orange}{rgb}{1,.5,0}
\definecolor{Violet}{rgb}{.5,0,.5}
\definecolor{Purple}{rgb}{.75,0,.25}
\definecolor{Brown}{rgb}{.75,.5,.25}
\definecolor{Grey}{rgb}{.7,.7,.7}
\definecolor{Black}{rgb}{0,0,0}

\definecolor{Red}{rgb}{1,0,0}
\definecolor{Blue}{rgb}{0,0,1}
\definecolor{Olive}{rgb}{0.41,0.55,0.13}
\definecolor{Yarok}{rgb}{0,0.5,0}
\definecolor{Green}{rgb}{0,1,0}
\definecolor{MGreen}{rgb}{0,0.8,0}
\definecolor{DGreen}{rgb}{0,0.55,0}
\definecolor{Yellow}{rgb}{1,1,0}
\definecolor{Cyan}{rgb}{0,1,1}
\definecolor{Magenta}{rgb}{1,0,1}
\definecolor{Orange}{rgb}{1,.5,0}
\definecolor{Violet}{rgb}{.5,0,.5}
\definecolor{Purple}{rgb}{.75,0,.25}
\definecolor{Brown}{rgb}{.75,.5,.25}
\definecolor{Grey}{rgb}{.7,.7,.7}
\definecolor{Black}{rgb}{0,0,0}

\newcommand{\ignore}[1]{{}}

\newcommand{\EE}{{\mathbb{E}}}
\newcommand{\PP}{{\mathbb{P}}}

\newtheorem{notation}[theorem]{Notation}

\begin{document}

\title{Harnack inequality for non-local Schr\"odinger 
operators.}

\author{
  Siva Athreya
  \thanks{8th Mile Mysore Road, Indian Statistical Institute,
         Bangalore 560059, India.
    Email: \url{athreya@ms.isibang.ac.in}}
  \and
  Koushik Ramachandra
  \thanks{Oklahoma State University, Stillwater, OK-74074 
    Email: \url{koushik.ramachandran@okstate.edu}}
  }


\maketitle

\begin{abstract}
Let $x \in \mathbb{R}^d$, $d \geq 3,$ and $f: \mathbb{R}^d \rightarrow \mathbb{R}$ be a twice
differentiable function with all second partial derivatives being
continuous. For $1\leq i,j \leq d$, let $a_{ij} : \mathbb{R}^d \rightarrow \mathbb{R}$
be a differentiable function with all partial derivatives being
continuous and bounded. We shall consider the Schr\"odinger operator
associated to
\begin{eqnarray*}
 \mathcal{L}f(x) &=& \frac12 \sum_{i=1}^d \sum_{j=1}^d
 \frac{\partial}{\partial x_i} \left(a_{ij}(\cdot) \frac{\partial
   f}{\partial x_j}\right)(x) + \int_{\mathbb{R}^d\setminus{\{0\}}}
      [f(y) - f(x) ]J(x,y)dy
 \end{eqnarray*}
where $J: \mathbb{R}^d \times \mathbb{R}^d \rightarrow \mathbb{R}$ is a symmetric measurable
function. Let $q: \mathbb{R}^d \rightarrow \mathbb{R}.$ We specify assumptions on $a,
q,$ and $J$ so that non-negative bounded solutions to $${\mathcal L}f
+ qf = 0$$ satisfy a Harnack inequality. As tools we also prove a Carleson estimate, a uniform Boundary Harnack Principle and a 3G inequality for solutions to ${\mathcal L}f = 0.$
\end{abstract}

\section{Introduction}

In this paper we study the Harnack inequality for solutions to the Schr\"odinger operator associated to  a specific class of non-local
operators.  Let $x \in \mathbb{R}^d$, $ d \geq 3$, $i,j \in \{1, \ldots, d\}$ and
$q,a_{ij}: \mathbb{R}^d \rightarrow \mathbb{R}$ and $J: \mathbb{R}^d \times \mathbb{R}^d \rightarrow
\mathbb{R}$. We consider positive bounded solutions to the Schr{\"o}dinger
equation, $\mathcal{L} u + q u = 0$ where
\begin{eqnarray}\label{L}
\mathcal{L}f(x) &=& \frac12 \sum_{i=1}^d \sum_{j=1}^d
\frac{\partial}{\partial x_i} \left(a_{ij}(\cdot) \frac{\partial
  f}{\partial x_j}\right)(x) + \int_{\mathbb{R}^d\setminus{\{0\}}}
     [f(y) - f(x) ]J(x,y)dy
  \end{eqnarray}
We show that when $q$ is in the Kato class, positive bounded
solutions to $\mathcal{L} u + q u = 0$ satisfy a Harnack inequality
(See Theorem \ref{Harnack} for the precise statement). 

For proving our main result we require several tools from the Potential
theory of ${\mathcal L}$.  In particular we prove a Carleson estimate
(Theorem \ref{Carleson}), a uniform Boundary Harnack Principle
(Theorem \ref{BdryHarnack}) and a 3G inequality (Proposition \ref{3G})
for solutions to ${\mathcal L}u = 0.$ In keeping with our objective we
have proved these results only for balls, the same proof should go
through in $C^{1,1}$ domains (in particular those domains as in
Notation \ref{not1}). To prove them we borrow techniques developed in
\cite{chensong2}, \cite{chensong}, \cite{cskv} and \cite{cskv2} for an
operator closely related to ${\mathcal L}$ (see (\ref{lla}).

The importance of the Harnack inequality and its implications to the
theory of elliptic and parabolic partial differential equations are
well known. We refer the interested reader to the survey article by
Kassmann \cite{kas}, and the references therein for a comprehensive
review of the classical Harnack inequality (i.e. for elliptic and
parabolic operators).

The Harnack inequality for solutions to (local) Schr\"{o}dinger
operators was first proved in \cite{AIS}, where they considered the
operator $\Delta + q,$ where $q$ is a given potential belonging to an
appropriate function space. It was shown in \cite{AIS} that
non-negative solutions to $\Delta u + qu = 0$ satisfy a Harnack
inequality. We refer the reader to \cite{CZ} for a detailed account of
the same. When $\Delta$ was replaced by a second order elliptic
operator in divergence form, the Harnack inequality was established
using analytic methods by \cite{cfg} and probabilistic methods in
\cite{cfz}. These would correspond to the case $J\equiv 0$, $a$ is
uniformly elliptic and bounded in (\ref{L}), and $q$ being in Kato
class (see Assumption (A) and (Q) below).

The simplest and most well studied pure jump process is the symmetric
stable process of index $\alpha$ (i.e. in our notation $J(x,y) =
\frac{c(d,\alpha)}{\mid x-y \mid^{d+\alpha} } $, and $a\equiv 0$,
${\mathcal L},$ reduces to $\Delta^{\frac{\alpha}{2}}$). This is a
Levy process whose infinitesimal generator is the fractional Laplacian
$\Delta ^{\alpha/2},$ defined by $$\Delta ^{\alpha/2}u(x) = c(d, \alpha)\lim_{\epsilon\to
  0}\int_{|y-x|>\epsilon} \dfrac{u(y)-u(x)}{\mid y-x\mid^{d+\alpha}}
dy,$$
\noindent where $0 < \alpha\leq 2,$ and $c(d,\alpha)$ is an
appropriate constant. The above limit exists if $u$ is
$C^2_b(\mathbb{R}^d)$ (the set of all bounded continuously twice
differentiable functions).  The Harnack inequality for non-negative
$\alpha$-harmonic functions was proved in the $1930$'s by M. Riesz by
using the corresponding Poisson Kernel representation (See
\cite{lan}). We refer the reader to \cite{lnm} for a thorough
introduction to the potential analysis of the fractional Laplacian and
other related operators.  \cite{Balev} consider general non-local
operators $\mathcal A$ defined by $$\mathcal Au(x)=
\int_{\mathbb{R}^d\setminus\{0\}}\left[u(x+h)- u(x) - 1_{\mid h\mid
    \leq 1}h.\nabla u(x)\right]\dfrac{J(x,x+h)}{\mid h\mid^{d
    +\alpha}}dh,$$ where $J$ is symmetric in $h$, uniformly bounded
above and below from $0.$ They prove a Harnack inequality holds for
non-negative $\mathcal A$-harmonic functions. 

The Harnack inequality for positive solutions to the Schr\"odinger
operator associated to $\Delta^{\frac{\alpha}{2}}$ has been proved in
\cite{bb} when $q$ is in the Kato-class associated with the pure-jump
process. To the best of our knowledge the Harnack inequality for
positive solutions to the Schr\"{o}dinger operator corresponding to
the general framework as in (\ref{L}) is not known. In areas such as
risk-sensitive control theory Harnack inequality for the Schr\"odinger
operator associated with (\ref{L}) is needed (see \cite{kumpal}). To
establish this inequality is the main purpose of this article.

A Harnack inequality was established in \cite{foon1} for positive
solutions to $\mathcal{B}u=0$,
\begin{eqnarray} \label{foonla}
 \mathcal{B}f(x) &=& \frac{1}{2}\sum_{i=1}^d \sum_{j=1}^d a_{ij}(x)
 \frac{\partial^2 f}{\partial{x_i}\partial{x_j}}(x) + \sum_{i=1}^d
 b_i(x) \frac{\partial f}{\partial{x_i}}(x) \nonumber \\ && +
 \int_{\mathbb{R}^d\setminus{\{0\}}} [f(x+h) - f(x) - 1_{\mid h\mid \leq
     1}h\cdot\nabla f(x)]J(x,x+h)dh,
 \end{eqnarray}
 with suitable assumptions on $a_{ij}$, $b_i$ and the kernel $J$. When
 ${\mathcal B}$ is in divergence form, (it is the same as ${\mathcal
   L}$ in (\ref{L})), a parabolic Harnack inequality (which implies
 the elliptic Harnack inequality) was estabilshed in \cite{CK2} for
 positive bounded solutions to ${\mathcal L} u = 0$ (See Proposition
 \ref{hq0} below). The non-local operator
$\mathcal L$ can be written as a sum $\mathcal L_c + \mathcal L_j$
where $\mathcal L_c$ is a differential operator, corresponding to the diffusion
part of the process and $\mathcal L_j$ an integral operator, corresponding to
the jump part of the same process.  The absence of scaling makes the
study of such processes difficult. When $J(x,y) =
\frac{c(d,\alpha)}{\mid x-y \mid^{d+\alpha }} $,and $a (\cdot)$ is the
constant function $a$, the operator reduces to ${\mathcal L}$ of the form
\begin{eqnarray} \label{lla}
{\mathcal L}  &=& a \Delta  +  \Delta^{\frac{\alpha}{2}}.  
\end{eqnarray}
Recently \cite{bkk} have proved a boundary Harnack inequality for
jump-type Markov processes on metric measure state spaces, under
comparability estimates of the jump kernel and Urysohn-type property
of the domain of the generator of the process.  The result holds for a
very general class of Markov process but does not include generators
in the general form  given by (\ref{L}).

The Markov processes, harmonic functions and Green function associated
with (\ref{lla}) have been well studied in a series of works by
several authors in \cite{chensong2}, \cite{chensong}, \cite{cskv} and
\cite{cskv2}. We shall use several of the results from these papers
and techniques inspired by these works. 

Throughout this paper all constants will be denoted by $c_1, c_2
\ldots$ They are all positive valued and their values are not
important. Their dependencies on parameters if needed will be
mentioned inside bracket for e.g $c_1(d)$. We will begin with
numbering afresh in each new result and proof.

\subsection{Preliminaries and Main result}\label{mrlt}

$\mbox{\hspace{1in}}$
\medskip

 For any $x \in \mathbb{R}^d$ and $r >0$, we set $\mid x \mid =
 \sqrt{\sum_{i=1}^d x_i^2} \mbox{ and }  B(x,r) = \{ y \in \mathbb{R}^d :\,
 \mid x -y \mid <r \}.$
 
 {\bf Assumption (A) : (Uniform ellipticity and boundedness) }{\it
    For all $i,j \in \{1,2, \ldots d\}$ the functions $a_{ij}$ are bounded and have
    continuous bounded partial derivatives. Furthermore, there exists $\lambda >0$ such that for every $\xi,$ $x$ in $\mathbb{R}^d$
\begin{equation*}
\lambda\mid \xi \mid^2 \leq \sum_{i=1}^d\sum_{j=1}^d a_{ij}(x) \xi_i
\xi_j\leq \frac{1}{\lambda}\mid \xi \mid^2  
\end{equation*}}

{\bf Assumption (J) :}{\it The function $J(\cdot,\cdot)$ is a non-negative, symmetric, measurable function such that there exist 
  $\kappa>0$ and $\alpha\in (0,2)$ such that:
\begin{equation}\label{dl2}
 \frac{\kappa}{\mid x-y \mid^{d+\alpha}}\leq J(x,y)\leq \frac{\kappa^{-1}}{\mid x-y \mid^{d+\alpha}}
\end{equation}
 for $x,y\in \mathbb{R}^d,$ $x\neq y.$}

We shall work with the conditional gauge to prove the main result.  To
ensure that the conditional gauge is bounded (Proposition
\ref{condgauge}) we shall follow \cite{chensong} and assume the
following about our function $q$.

{\bf Assumption (Q) (Kato class):}  $q: \mathbb{R}^d\rightarrow\mathbb{R}$ is measurable, and satisfies:
\begin{equation} \label{kX}
\lim_{r \downarrow 0} \sup_{x \in \mathbb{R}^d} \int_{\mid y -x \mid < r} dy \frac{\mid q(y)\mid}{\mid x-y\mid^{d-2}  } =0
\end{equation} 
We are now ready to state our main result.
\begin{theorem}\label{Harnack}
  Let $R\in(0, \frac{1}{2}],$ $x_0 \in \mathbb{R}^d$. Let $\mathcal L$ be the
    nonlocal operator defined in (\ref{L}) and $q: \mathbb{R}^d \rightarrow
    \mathbb{R}$. Assume (A), (J), and (Q). Suppose that $u\in C^2_b(\mathbb{R}^d)$ is
    strictly positive and satisfies
    \begin{equation*}
\mathcal Lu(x) + q(x)u(x) = 0
\end{equation*} for $x \in B(x_0,R)$. Then, there exists a positive constant $c_1 \equiv c_1(R, q, \kappa, d, \alpha,\lambda)$ such that
    \begin{equation} \label{HI} u(x) \leq c_1 u(y), \end{equation}
for all $x, y \in B(x_0, R/2).$
\end{theorem}

\begin{remark} Despite the operator in (\ref{L}) being  inhomogeneous in space,  the constant $c_1 \equiv c_1(R, q, \kappa, d, \alpha,\lambda)$ is independent of
 $x_0$ (See Remark \ref{r:constant} after proof of Theorem \ref{Harnack}).  By the standard chain of balls  argument, it is easy to see  that the Harnack inequality also holds in
 any ball $B(x_0, R)$ for all $R >0$ with the appropriate constant
 $c_1$ depending on $ q$ and $R.$
\end{remark}

\begin{remark} We did not consider the Schr\"odinger operator
associated to ${\mathcal B}$ as in (\ref{foonla}). It is perhaps
possible to state a version of Theorem \ref{Harnack} for ${\mathcal
  B}$, by assuming that $q$ satisfies an abstract condition involving
the Green function along with the additional assumptions in
\cite{foon1}. We wanted a verifiable condition like the one mentioned in
Assumption (A) so we restricted our attention to ${\mathcal L}.$
 \end{remark}

Fix $0 < R < \frac{1}{2}$, $x_0 \in \mathbb{R}^d$ and let $B$ denote $B(x_0,R)$ for the rest of the paper. 

\subsection{Proof of Theorem \ref{Harnack}} \label{pfHarnack}

$\mbox{\hspace{1in}}$
\medskip

For the case $J \equiv 0,$ Theorem \ref{Harnack} was proved in \cite{cfz} (Theorem 5.1). We will follow the ideas in \cite{cfz} for proving Theorem \ref{Harnack}.  Below we state several propositions whose proof we shall
provide in subsequent sections. Assuming these we shall first prove Theorem \ref{Harnack}.

Under assumptions (A) and (J) there is a symmetric strong Markov
process, $(\PP^x,X_t)$, with c\`adl\`ag paths
associated with $\mathcal L$ (See \cite{CK2}).  We will denote the
Green function of the process by $G$, i.e. a Borel function $G(x,y)$
on $\mathbb{R}^d\times \mathbb{R}^d$ such that
\begin{equation}
\label{green}
\EE^x \left ( \int_0^\infty f(X_s) ds \right)  = \int_{\mathbb{R}^d}f(y) G(x,y) dy
\end{equation}
for all bounded measurable $f: \mathbb{R}^d \rightarrow \mathbb{R}.$
\begin{notation}
For any Borel set $A$,
 $$T_A = \inf \{ t \geq 0 : X_t \in A\} \mbox{ and } \tau_A = \inf \{ t \geq 0 : X_t \not \in A\}$$ denote the  hitting time and  exit
time of set $A$.
\end{notation}

\vspace{0.1in}

\noindent It is standard to note that if $u\in C^2_b(\mathbb{R}^d)$ and satisfies
${\mathcal L}u =0$ then via Ito's formula $u(X_t)$ is a martingale.
\begin{definition} 
Let $D$ be a bounded domain in $\mathbb{R}^d.$ We say that a
measurable function $h :\mathbb{R}^d \rightarrow \mathbb{R}$ is $\mathcal L-$harmonic
in $D$ if for every open set $U$ such that $U\subset\bar U\subset D,$
\begin{equation}\label{harmonicdefn}
h(x) = \mathbb{E}^x(h(X_{\tau_U})),
\end{equation}
for $x\in U.$ We say that $h$ is regular harmonic in $D$ if it is
harmonic in $D$ and in addition, the relation \eqref{harmonicdefn}
holds with $D$ replacing $U.$
\end{definition}
For the rest of the paper we write $\mathcal Lu = 0,$ to mean that
$u$ is $\mathcal L$-harmonic. In \cite{CK2}, a parabolic Harnack
inequality was proven for solutions to ${\mathcal L} u = 0.$ This will
imply the  elliptic Harnack inequality (alternative references are \cite{foon1}, \cite{foon2}). We state this next.
\begin{proposition} \label{hq0}
Assume (A) and (J). Let $x_0\in \mathbb{R}^d$ and $R\in(0,1/2]$. Suppose $v$
  is nonnegative and bounded on $\mathbb{R}^d$ and ${\mathcal L} v = 0$ in
  $B.$ Then there exists a positive constant $c_1 = c_1(R,q, \kappa, d, \alpha, \lambda)$ (in particular independent of $x_0$ and $v$) such that
  \begin{equation}
 v(x)\leq c_1 v(y),
\end{equation}
whenever $x,y \in B(x_0, \frac{R}{2})$
\end{proposition}

In what follows below we always assume that $q$ satisfies assumption $(Q).$  To work with probabilistic solutions for the Schr\"odinger equation
 we will need to show that the ball is gaugeable, whose definition we now give.  For $ t \geq 0,$ let $$e_q(t) =  \exp\left(\int_{0}^{t}q(X_s)ds\right).$$ 
\begin{definition} Let $U$ be an open subset of the ball $B$.
The function $H: U \rightarrow \mathbb{R} \cup \{ \infty \}$ given by
\begin{equation}H(x) =
\mathbb{E}^x\left(e_q(\tau_U)\right)
\end{equation}
 is called the gauge for $(U, q).$ If the gauge $H$ is bounded in $U$ we will call $(U, q)$ gaugeable.
\end{definition}
We now state the Feynman-Kac representation for
solution $u\in C^2_b(\mathbb{R}^d)$ of ${\mathcal L}u + qu =0$ provided the ball is gaugeable.
 \begin{proposition}\label{fkac}
Let $q$ satisfy assumption (Q) and suppose $(B, q)$ is gaugeable. If
$u\in C^2_b(\mathbb{R}^d)$ solves $\mathcal Lu + qu = 0$ in $B,$ then for all
$x \in B$
\begin{equation}\label{FC}
u(x) = \mathbb{E}^x\left (e_q(\tau_B)u(X_{\tau_{B}})\right)
\end{equation}
\end{proposition}
We  provide a sufficient condition for the ball $B$ to be gaugeable
\begin{proposition}\label{ballgauge}
Let $q$ satisfy assumption $(Q).$ Let $u$ be a bounded solution of $\mathcal{L}u + qu = 0$ in
$B,$ with $\inf_{B}u > 0.$ Then $(B, q)$ is gaugeable.
\end{proposition}

The next ingredient required in the proof will be to condition on the exit measure and to employ the conditional gauge theory from the literature. For $x,y \in B$, let  $\PP^x_y$ 
and $\mathbb{E}^{x}_{y}$ denote the probability and expectation for
the Doob's $h$-transformed process of $X$ with $h(\cdot) = G_B(\cdot,y),$ where $G_B$ denotes the Green function of the process $X_t$ killed on exiting $B.$ More precisely, $G_B$ is defined by
\begin{equation}
\label{green2}
\mathbb{E}^x \left (\int_{0}^{\tau_B} f(X_s) ds \right)  = \int_{B}f(y) G_B(x,y) dy
\end{equation}
for all bounded measurable $f: \mathbb{R}^d \rightarrow \mathbb{R}.$

\begin{proposition}\label{conditional}
Let $\phi$ be a non-negative $\mathcal{F}_{\tau_B-}$ measurable function and $A$ be a Borel subset of $B^c.$ Then,
$$\mathbb{E}^x(\phi; X_{\tau_B}\in A) = \mathbb{E}^x(\mathbb{E}^x_{X_{\tau_B-}}(\phi); X_{\tau_B}\in A)$$
\end{proposition}
When ${\mathcal L } = \Delta$, \cite{CZ} contains a proof of the above
and it explicitly uses the density of the harmonic measure of the
process. When ${\mathcal L } = \Delta^\alpha$, \cite{bb} contains a
proof of the above by using the joint density of $(X_{\tau_B-},
X_{\tau_B})$. We had to combine both these aspects: when the process
exits the ball via the boundary of the ball  we show existence of
a Martin Kernel and prove a density for the Harmonic measure in
Theorem \ref{density}; and when the process exits the ball into the
complement via jump we use the Levy system formula (see (1.4) in \cite{cskv}), thus establishing Proposition \ref{conditional}.

Another key step required is to verify the boundedness of the conditional gauge.
\begin{definition} \label{cgauge}For any $x \in B, y \in \bar{B}, x \neq y$, the conditional
  gauge is defined to be
\begin{equation}
F(x,y) = \mathbb{E}^{x}_{y}\left(e_q(\tau_{B\setminus \{y\}})\right). 
\end{equation}
\end{definition}
We shall establish the following result about the conditional gauge.
\begin{proposition}\label{condgauge}
Let $q$ satisfy assumption (Q). Then, either $F\equiv\infty,$ or there exist positive constants
$c_1$ and $c_2$ such that
\begin{equation}\label{bounds}
c_1 \leq F(x,y)\leq c_2, \quad x\in B, y\in \bar{B}, x \neq y.\end{equation}
 \end{proposition}

We now have all the ingredients to prove the main result.

\vspace{0.1in}

\noindent \textbf{Proof of Theorem \ref{Harnack}:} From the hypothesis
and Proposition \ref{ballgauge} we know that $(B, q)$ is
gaugeable. So, $$\mathbb{E}^x(e_q(\tau_B)) < \infty.$$ This implies
that the conditional gauge $F(\cdot, \cdot)$ cannot be identically
infinity. Therefore by Proposition \ref{condgauge}, $F$ satisfies
(\ref{bounds}). From Proposition \ref{fkac} we know that the solution to the
Schrodinger equation $u$ satisfies $$u(x) =
\mathbb{E}^x\left(e_q(\tau_B)u(X_{\tau_B})\right),$$ for all $x \in
B$. For $x,y \in B(x_0, R/2),$ we have
\begin{eqnarray} \label{CFZ1}
u(x)&=& \mathbb{E}^x\left(e_q(\tau_B)u(X_{\tau_B})\right) \nonumber\\
&=& \mathbb{E}^x\left(\mathbb{E}^x_{X_{\tau_B-}}(e_q(\tau_B))u(X_{\tau_B})\right).
\end{eqnarray}
by Proposition \ref{conditional}.  Then using the upper bound for the conditional gauge from (\ref{bounds}), in (\ref{CFZ1}) we have
\begin{eqnarray} \label{CFZ2}
u(x)&\leq & c_2 \mathbb{E}^x\left(u(X_{\tau_B})\right)
\end{eqnarray}
In (\ref{CFZ2}), applying the Harnack inequality for the
$\mathcal L$ harmonic function $v(x) = \mathbb{E}^x\left(u(X_{\tau_B})\right)$, (see Proposition \ref{hq0}) we have
\begin{eqnarray} \label{CFZ3}
u(x)&\leq & c_3\mathbb{E}^y\left(u(X_{\tau_B})\right) \end{eqnarray}
We now reverse the estimate, using the lower bound from (\ref{bounds}) and Proposition \ref{conditional} in (\ref{CFZ3}) we have
\begin{eqnarray}
&\leq & \frac{c_3}{c_4}\mathbb{E}^y\left(\mathbb{E}^y_{X_{\tau_B-}}(e_q(\tau_B))u(X_{\tau_B})\right)\nonumber\\
&=& c_5\mathbb{E}^y\left(e_q(\tau_B)u(X_{\tau_B})\right)\nonumber\\
&=&c_5 u(y).
\end{eqnarray} 
This finishes the proof of the theorem.
\qed

\begin{remark} \label{r:constant} We note that in the proof above the constants in (\ref{CFZ1}), (\ref{CFZ2}), and  (\ref{CFZ3}) do not depend on $x_0$. Consequently the constant  $c_1 \equiv c_1(R, q, \kappa, d, \lambda)$ in the statement of Theorem \ref{Harnack} does not depend $x_0$. 
  \end{remark}

{\bf Layout:} The rest of the paper is organized as follows.  In the
next section, Section \ref{sgauge}, we prove Proposition \ref{fkac} and Proposition
\ref{ballgauge}.  The remainder of the
paper is devoted to proving Proposition \ref{conditional} (proved in
Section \ref{GFE}) and Proposition \ref{condgauge} (proved in Section
\ref{MBDH}). These propositions require three key results from the
potential theory of ${\mathcal L}$. For this, in Section \ref{potl}, we prove a
Carelson estimate (Theorem \ref{Carleson}) and a uniform Boundary
Harnack Principle (Theorem \ref{BdryHarnack}) followed by the 3G
inequality (Proposition \ref{3G}) in Section \ref{GFE}.  Results on
the Martin kernel and the Martin Boundary along with the density of
the Harmonic measure are contained in Section \ref{MBDH}. The results
in these sections are of independent interest as well.

\section{Gauge and Feynman-Kac Representation}  \label{sgauge}

In this section we prove Proposition \ref{fkac} and Proposition
\ref{ballgauge}.  For any Borel set $A,$ let $$T_A = \inf\{t: X_t\in
A\} \mbox{ and } \tau_A = \inf\{t: X_t\in A^c\}$$ be the hitting time
and the exits times from $A$ respectively. From Lemma 2.6 in
\cite{CK2}), we know that for all $x \in B$ there exists $c_1 >0$ such that
\begin{equation}\label{fupperbd}
\mathbb{E}^x(\tau_{B})\leq c_1R^2.
\end{equation}
Hence for any Borel set $A \subset B$, $\tau_A$ is finite almost
surely $\PP^x$ for $x \in A$.
\begin{definition}
The function $H: B \rightarrow \mathbb{R} \cup \{ \infty \}$ given by
\begin{equation}H(x) =
\mathbb{E}^x\left(\exp\left(\int_{0}^{\tau_B}q(X_s)ds\right)\right)
\end{equation}
 is called the gauge for $(B, q).$ If the gauge $H$ is bounded in $B,$ we will call $(B, q)$ 
gaugeable.
\end{definition}
 We shall first prove that if $q$ satisfies $(Q)$ then every sufficiently `small' set is gaugeable. Let $m$ denote the Lebesgue measure on $\mathbb{R}^d$.
\begin{lemma} {(\bf Gaugeable sets)} 
Let $q$ satisfy assumption {(Q)}. Then, there exists $\delta>0$ such
that for every ball $C\subset B,$ with $m(C) <\delta,$ we have $(C,q)$
is gaugeable.
\end{lemma}
\begin{proof}
Let $C$ be a ball with $m(C) < \delta$ and $ x \in C$. Then using the definition of $G_B$ and the upper bound  (\ref{upper}) given by Lemma \ref{gbounds} below, we have
\begin{eqnarray*}
\mathbb{E}^x\left(\int_{0}^{\tau_C}\mid q(X_s)\mid ds\right) &=&
\int_{C}G_B(x,y)\mid q(y)\mid dy\\& \le &
c_1\int_{C}\mid x-y\mid^{2-d}\mid q(y)\mid dy\\ & = & \eta < 1,
\end{eqnarray*}
if $\delta$ is small enough, by assumption $(Q).$ By a standard application of Khasminki's lemma we have
\begin{equation}\label{finiteE}
\sup _{x \in C} \mathbb{E}^x\left(\exp\left(\int_{0}^{\tau_C}\mid q(X_s)\mid ds\right)\right)\leq\frac{1}{1-\eta} < \infty.
\end{equation}
This proves that $(C,q)$ is gaugeable. 
\end{proof}
We next prove a Feynman-Kac representation for solutions $u$ to
${\mathcal L}u + qu = 0$.

\vspace{0.1in}

\noindent \textbf{Proof of Proposition \ref{fkac}:}
For $t>0,$ let $Y_t = \int_{0}^{t}q(X_s)ds,$ $V_t = \exp(Y_t)$ and
$W_t = u(X_t).$ Note that $[V, W]_t = 0$ (since $V_t$ is
continuous). Applying Ito's product formula to $V,W$ (see \cite{Bass}) and taking expectations we have
\begin{equation}\label{ito}
\mathbb{E}^x e^{Y_{t\wedge\tau_B}}u(X_{t\wedge\tau_B})= u(x) +
\mathbb{E}^x \int_{0}^{t\wedge\tau_B}u(X_s)e^{Y_s}q(X_s)ds +
\mathbb{E}^x \int_{0}^{t\wedge\tau_B} e^{Y_s}\mathcal Lu(X_s)ds.
\end{equation}
\noindent Since $\mathcal Lu + qu =0,$ this implies
\begin{equation}\label{2a}
\mathbb{E}^x e^{Y_{t\wedge\tau_B}}u(X_{t\wedge\tau_B})= u(x).
\end{equation}
 As $(B,q)$ is gaugeable, allowing $ t \rightarrow \infty$ in (\ref{2a}), the
dominated convergence theorem implies (\ref{FC}). 
\qed

\vspace{0.1in}

\noindent When $X_t$ is a Brownian motion, it is known that the union of
gaugeable balls is gaugeable (see Lemma 4.16 in \cite{CZ}). The result
is true for solutions to the martingale problem as
well. The proof is similar and requires only minor modification. For
the sake of completeness we state the result and prove it below.

\begin{proposition} \label{unionG}
Let $C_1, C_2 \subset B$ be balls with $C_1\cap C_2\neq\varnothing$
and suppose $(C_i, q)$ is gaugeable. Let $C = C_1\cup C_2.$ Suppose
$q$ satisifies assumption $(Q),$ and there exists a bounded solution $u$ satisfying
$\mathcal{L}u + qu = 0$ in $C,$ with $\inf_{C}u > 0.$ Then $(C, q)$ is
gaugeable.
\end{proposition} 

\begin{proof}  Define for any $t >0$,  $e_q(t) = \exp(\int_{0}^{t}q(X_s)ds).$
Since each $C_i,$ $i=1,2$ is gaugeable, we can apply Proposition \ref{fkac} to observe that for $i =1,2$ and $x \in C_i$
\begin{equation}\label{star}
u(x) = \mathbb{E}^x\left(e_q(\tau_{C_i})u(X_{\tau_{C_i}})\right).
\end{equation} 
 We will show that equation (\ref{star}) holds when $C_i$ is replaced
 by $C.$ Without loss of generality we may assume $x\in C_1.$ Let $T_0
 = 0,$ and for $n\geq 1$ set
\begin{eqnarray}\label{times}
T_{2n-1}& =& T_{2n-2} + \tau_{C_1}\circ\theta_{T_{2n-2}} \nonumber\\
T_{2n} & =& T_{2n-1} + \tau_{C_2}\circ\theta_{T_{2n-1}},
\end{eqnarray}
 where $\theta$ is the canonical time-shift operator. We will
show that
\begin{equation}\label{fatoulim}
u(x) = \mathbb{E}^x\{e_q(T_m))u(X_{T_m})\}, \hspace{0.05in} m\geq 0.
\end{equation}
This is true for $m=0.$ Suppose that $u(x)= \mathbb{E}^x\left(e_q(T_{2n})u(X_{T_{2n}})\right).$ Then on $\{T_{2n} < \tau_C\},$ we have $X_{T_{2n}}\in C_1.$ Therefore by (\ref{star}), 
\begin{eqnarray}\label{star1}
u(x) &= &\mathbb{E}^x\left(T_{2n} = \tau_C,  e_q(T_{2n})u(X_{T_{2n}})\right)\nonumber\\
     & +& \mathbb{E}^x\left(T_{2n} < \tau_C, e_q(T_{2n})\mathbb{E}^{X_{T_{2n}}}\left(e_q(\tau_{C_1})u(X_{\tau_{C_1}})\right)\right).
\end{eqnarray}
On $\{T_{2n} = \tau_C\},$ we have $T_{2n} = T_{2n+1}.$ Hence the first
term on the right side of equation (\ref{star1}) is equal to
$\mathbb{E}^x\left(T_{2n} = \tau_C, e_q(T_{2n+1})u(X_{T_{2n+1}})\right). $ By
the definition of $T_{2n+1}$ and the strong Markov property, the
second term on the right side of equation \eqref{star1} is equal to\newline
$\mathbb{E}^x\left(T_{2n} < \tau_C, e_q(T_{2n+1})u(X_{T_{2n+1}})\right).$ Now 
adding the two terms we obtain \eqref{fatoulim} with $m= 2n +1.$ In a
similar manner, one can prove that if \eqref{fatoulim} holds for $m
=2n + 1,$ then it holds for $m= 2n + 2.$ Hence \eqref{fatoulim} holds
for all $m\geq 0$ by induction and it implies
that \begin{equation} \label{ubound}\left[\inf_{x\in C}
    u(x)\right]\mathbb{E}^x\left(e_q(T_m)\right)\leq\sup_{x\in C} u(x).
\end{equation}
We now establish that almost surely, $\lim_{m\to\infty}T_m = \tau_C.$
By \eqref{fupperbd} above, we first note that $\tau_C <
\infty$ with probability $1.$ As $T_m$ is increasing in $m$ and
$T_m\leq\tau_C$ we have $\lim_{m\to\infty}T_m = T\leq\tau_C.$ If $T <
\tau_C,$ then since $X_{T_{2n - 1}}\in\partial C_1$ and
$X_{T_{2n}}\in\partial C_2$ for all $n,$ we can use the fact that the
process has left limits and conclude that
$$X_T\in \partial C_1\cap\partial C_2\subset (C_1\cup C_2)^c.$$
But this implies that $T\geq\tau_C,$ which is a contradiction. Therefore almost surely $\lim_{m\to\infty}T_m = \tau_C.$ 
Using Fatou's lemma and the fact that $u$ is strictly positive, from \eqref{ubound} we have that $$\mathbb{E}^x\left(e_q(\tau_C)\right)\leq\frac{\sup_{x\in C} u(x)}{\inf_{x\in C} u(x)} < \infty.$$ So $(C, q)$ is gaugeable. 
\end{proof}

We have seen that every ball of sufficiently small radius is gaugeable. Now we prove Proposition \ref{ballgauge} which states that the ball $B$ is gaugeable.

\vspace{0.1in}

\noindent \textbf{Proof of Proposition \ref{ballgauge}:}
There exists a a sequence of bounded domains $\{D_n\}$ such that
$\overline{D_n}\subset B,$ $D_n\uparrow B.$ Recall that  $m$ denotes the Lebesgue measure, so it has no
atoms.  Therefore for each $n,$ $D_n$ can be written as the finite
union of balls $C,$ with $m(C) <\delta$ so that each $(C, q)$ is
gaugeable. Then by repeated application of Proposition \ref{unionG},
$(D_n, q)$ is gaugeable for each $n,$ and by Proposition \ref{fkac},
$$ u(x) = \mathbb{E}^x\left(e_q(\tau_{D_n})u(X_{\tau_{D_n}})\right),\hspace{0.05in} x\in D_n.$$ 
As before, we have $$ \mathbb{E}^x(e_q(\tau_{D_n})\leq\frac{\sup_{x\in B}u(x)}{\inf_{x\in B}u(x)}$$
Since $\tau_{D_n}\uparrow\tau_B < \infty$ a.s, we obtain by Fatou's
lemma that
$$\mathbb{E}^x(e_q(\tau_{B}))\,\,\leq\,\,\liminf_{n \rightarrow \infty}\mathbb{E}^x\left(e_q(\tau_{D_n})\right)\,\, \leq \,\,\frac{\sup_{x\in B}u(x)}{\inf_{x\in B}u(x)}\,\,<\,\,\infty.$$
\qed

\section{Potential Theory for $\mathcal{L}$} \label{potl} 

In this section we state and prove a uniform Boundary Harnack
principle for $\mathcal L$-harmonic functions (Theorem
\ref{BdryHarnack}). The classical version of
Boundary Harnack principle follows from this result. A key ingredient to
prove the Boundary Harnack principle is the Carleson Estimate (Theorem
\ref{Carleson}), which we prove first.  We begin by fixing some
notation.
\begin{notation} \label{not1}
The ball $B$ is a smooth domain. So there exists a localization radius $R_1 < R/4$ and a constant $M_1$ such that for every
$Q\in\partial B,$ there exist a smooth function $\phi = \phi_Q :
  \mathbb{R}^{d-1}\rightarrow\mathbb{R}$ satisfying $\phi(0) = 0,
  \nabla \phi(0) = 0,$ $\mid \nabla\phi(x) - \nabla\phi(y)\mid \leq M_1\mid x
  -y\mid ,$ and a coordinate system $CS_Q$ with $y = (\tilde{y}, y_d)$
  with its origin at $Q$ such that
$$B(Q, R_1)\cap B = \{(\tilde{y}, y_d)\in B(0, R_1): y_d >
  \phi(\tilde{y})\}.$$ We define $\rho_Q(y)= y_d - \phi(\tilde y)$,
  denote $\delta_B(x) = dist(x , \partial B),$ $r_0 =
  \frac{R_1}{4(1+M_1^2)},$ $R_0 = \frac{R_1}{\sqrt{1+ M_1^2}}$ and
  $$D_Q(r_1, r_2) = \{y\in B: 0 < \rho_Q(y) < r_1, \mid \tilde y\mid <
  r_2\}.$$
\end{notation}
An important ingredient in the proofs will
be the Levy system formula (see (1.4) in \cite{cskv}) associated with
the jump process $X$ given by a jump kernel $J$. For any non-negative
measurable function $g: \mathbb{R}_{+}\times\mathbb{R}^d\times\mathbb{R}^d \rightarrow [0, \infty),$ with $g(s, y, y)
= 0$ for all $y\in\mathbb{R}^d,$ any stopping time $T$ (with respect
to filtration of X), and any $x\in\mathbb{R}^d$
\begin{equation}\label{levy}
\mathbb{E}^x\left(\sum_{s\leq T}g(s, X_s, X_{s-})\right) =
\mathbb{E}^x\left(\int_{0}^{T}\left(\int_{\mathbb{R}^d}g(s, X_s,
  y)J(X_s,y)dy\right)ds\right)
\end{equation}  
We remark at this point that the proofs of the results in this entire
section follow the noation, ideas, and techniques in
\cite{cskv}. Instead of citing the results without proof when
required, we reproduce the proof here for the reader's convenience.

\subsection{Carleson Estimate}

We begin with some technical lemmas from the literature to understand
the behavior of exit distributions of the process $X_t$.
\begin{lemma}\label{foon2}
There exists positive constants $c_1$ and $c_2$ such that
\begin{equation}\label{foonL}
\mathbb{E}^x(\tau_{B})\leq c_1R^2, \quad x\in B,
\end{equation}
and 
\begin{equation}\label{foonLL}
\mathbb{E}^x(\tau_{B})\geq c_2R^2,\quad  x\in B(x_0, \frac{R}{2}).
\end{equation}
\end{lemma}

\begin{proof} See Lemma 3.4 in \cite{foon1}.
\end{proof}
\begin{lemma}\label{fo1}
There exists a non-decreasing function $\psi:(0, 1)\rightarrow (0,
1),$ such that if $C\subset B(x_0, r),$ $\mid C\mid  >0,$ $r\in(0, 1],$ and
  $x\in B(x_0, r/2),$ then
$$\mathbb{P}^x(T_C\leq\tau_{B(0, r)})\geq\psi\left(\frac{\mid C\mid }{r^d}\right).$$
\end{lemma}
\begin{proof} See Corollary 4.9 in \cite{foon1}.
\end{proof}
\begin{lemma}\label{hmeasure}
Let $R_1$ and $\rho_Q$ be defined as above.  Then, there exists a constant $\delta = \delta(R_1, M_1) >
0,$ such that for all $Q\in\partial B, x\in B$ with $\rho_Q(x) <
R_1/2,$
$$\mathbb{P}^x(X_{\tau(x)}\in B^c)\geq\delta,$$
where $\tau(x) = \inf\{t>0: X_t\notin B\cap B(x, 2\rho_Q(x))\}.$
\end{lemma}
\begin{proof}
Denote $\widetilde B = B(x, 2\rho_Q(x)),$ and $C = \widetilde B\cap
B^c.$ Recall that $2\rho_Q(x) < R_1 < 1.$ Observe that $\{T_C <
\tau_{\widetilde B}\}\subset\{X_{\tau(x)}\in B^c\}.$ Therefore using
Lemma \ref{fo1} (with $r = 2\rho_Q(x)$), we obtain
\begin{equation}\label{0}
\mathbb{P}^x(X_{\tau(x)}\in B^c)\geq \mathbb{P}^x(T_C < \tau_{\widetilde
  B})\geq\psi(1/2) > 0.
\end{equation}
This shows that we can take $\delta = \psi(1/2),$ and finishes the proof of the lemma.
\end{proof}
We are now ready to prove the Carleson estimate.
\begin{theorem}\label{Carleson}
Let $Q\in\partial B. $ Let $u$ be a non-negative function
in $\mathbb{R}^d$ that is $\mathcal L$ harmonic in $B\cap B(Q, r),$
with $r< R_1/2$ and suppose that $u$ vanishes continuously on $B^c\cap
B(Q, r).$ Then there exists a positive constant $c= c(\alpha,
R_1, M_1)$ such that
\begin{equation}\label{estimate}
 u(x)\leq cu(x_0), \quad \mbox{for}\quad x\in B\cap B(Q, r/2).
\end{equation}
where $x_0\in B\cap B(Q, r),$ with $\rho_Q(x_0) = r/2.$
\end{theorem}
\begin{proof}
Since $r < R_1/2,$ by the Harnack inequality and a chain argument, it
is sufficient to prove \eqref{estimate} for $x\in B\cap B(Q, r/12)$
and $\tilde{x_0} = \tilde{Q}.$ We may also normalize so that $u(x_0) =
1.$ In the following proof, all the constants $\delta, \beta, \eta,$
and $c_i$ are always independent of $r.$ First choose $0 < \gamma <
\alpha/(d + \alpha)$ and let
$$ B_0 = B\cap B(x, 2\rho_Q(x)), \quad B_1 = B(x, r^{1-\gamma}\rho_Q(x)^{\gamma}).$$
\noindent We then set 
$$B_2 = B(x_0, \rho_Q(x)/3), \quad B_3 = B(x_0, 2\rho_Q(x)/3)$$
and 
$$\tau_0 = \inf\{t>0 : X_t\notin B_0\}, \quad \tau_2 = \inf\{t>0 : X_t\notin B_2\}.$$
By Lemma \ref{hmeasure}, there exists $\delta = \delta(R_1, M_1)$ such that 
\begin{equation}\label{1a}
\mathbb{P}^x(X_{\tau_0}\in B^c)\geq\delta, \quad x\in B(Q, r/4).
\end{equation}
By the Harnack inequality and a chain argument, there exists $\beta$ such that 
\begin{equation}\label{2}
u(x) < \left(\frac{\rho_Q(x)}{r}\right)^{-\beta}u(x_0), \quad x\in B(Q, r/4).
\end{equation}
Since $u$ is $\mathcal L$-harmonic in $B_0,$ we may write 
\begin{equation}\label{3}
u(x) = \mathbb{E}^x\left(u(X_{\tau_0}); X_{\tau_0}\in B_1\right) + \mathbb{E}^x\left(u(X_{\tau_0}); X_{\tau_0}\notin B_1\right), \quad x\in B(Q, r/4).
\end{equation}
We will assume the following Lemma and complete the proof of the Theorem.
\begin{lemma} \label{4lemma} There exists $\eta >0$ such that 
\begin{equation}\label{4}
\mathbb{E}^x\left(u(X_{\tau_0}); X_{\tau_0}\notin B_1\right)\leq u(x_0) \hspace{0.07in}\mbox{if}\hspace{0.05in} x\in B\cap B(Q, r/12), \mbox{ and }  \rho_Q(x) < \eta r. 
\end{equation}
\end{lemma}

We will prove the Carleson estimate by contradiction. Recall that
$u(x_0) = 1.$ Suppose that there exists $x_1\in B(x, r/12)$ such that
\begin{equation}\label{ce1}u(x_1)\geq K > \eta^{-\beta}\vee (1+\delta^{-1}),\end{equation}
where $K$ is a constant that will be specified later. By \eqref{2} and the assumption that $u(x_1)\geq K > \eta^{-\beta},$ we have $(\rho_Q(x_1)/r)^{-\beta} >
u(x_1)\geq K > \eta^{-\beta}.$ Hence $\rho_Q(x_1) < \eta r.$ Let $B_0,
B_1$ and $\tau_0$ now be defined with respect to the point $x_1$
instead of $x.$ Then by \eqref{3}, \eqref{4} and $K> (1+\delta^{-1}),$
we have
$$K\leq u(x_1)\leq\mathbb{E}^{x_1}\left(u(X_{\tau_0}); X_{\tau_0}\in B_1\right) +
1,$$ and hence, using \eqref{ce1},
$$\mathbb{E}^{x_1}\left(u(X_{\tau_0}); X_{\tau_0}\in B_1\right)\geq u(x_1) -1 >
\frac{1}{1+\delta}u(x_1). $$ If
$K\geq 2^{\frac{\beta}{\gamma}},$ then $B^c\cap B_1\subset B^c\cap
B(Q, r).$ By using the assumption that $u=0$ on $B^c\cap B(Q, r)$ and \eqref{1a}we have
\begin{eqnarray}
\mathbb{E}^{x_1}\left(u(X_{\tau_0}); X_{\tau_0}\in B_1\right)&=& \mathbb{E}^{x_1}\left(u(X_{\tau_0}); X_{\tau_0}\in B_1\cap B\right)\nonumber\\
&\le& \mathbb{P}^{x}(X_{\tau_0}\in B)\sup_{x \in B \cap B_1}u \nonumber\\
&\le& (1-\delta)\sup_{x \in B \cap B_1}u.
\end{eqnarray}
\noindent Therefore, $\sup_{x \in B \cap B_1}u(x) > u(x_1)/((1+\delta)(1-\delta)),$ i.e, there exists a point $x_2\in B$ such that 
$$\mid x_2- x_1\mid \leq r^{1-\gamma}\rho_Q(x_1)^{\gamma}\hspace{0.05in}\mbox{and}\hspace{0.05in} u(x_2) > \frac{1}{1-\delta^2}u(x_1)\geq\frac{1}{1-\delta^2}K.$$
\noindent By induction, if $x_k\in B\cap B(Q, r/12)$ with $u(x_k)\geq K/(1-\delta^2)^{k-1}$ for $k\geq 2,$ then there exists $x_{k+1}\in B$ such that 
\begin{equation}\label{11}
\mid x_k- x_{k+1}\mid \leq
r^{1-\gamma}\rho_Q(x_1)^{\gamma}\hspace{0.05in}\mbox{and}\hspace{0.05in}
u(x_{k+1}) > \frac{1}{1-\delta^2}u(x_k)\geq\frac{1}{(1-\delta^2)^k}K.
\end{equation}
From \eqref{2} and \eqref{11}, it follows that $\rho_Q(x_k)/r\leq (1-\delta^2)^{(k-1)/\beta)}K^{-1/\beta}$ for every $k\geq 1.$ Therefore,
\begin{eqnarray}\label{12}
\mid x_k- Q\mid &\le& \mid x_1- Q\mid  + \sum_{j=1}^{k-1}\mid x_{j+1}- x_j\mid \leq r/12 + \sum_{j=1}^{\infty}r^{1-\gamma}\rho_Q(x_j)^{\gamma} \nonumber\\
&\le& r/12 + r^{1-\gamma}\sum_{j=1}^{\infty}(1-\delta^2)^{(j-1)\gamma/\beta}K^{-\gamma/\beta}r^{\gamma}\nonumber\\
&=& r/12  + r^{1-\gamma}r^{\gamma}K^{-\gamma/\beta}\sum_{j=1}^{\infty}(1-\delta^2)^{(j-1)\gamma/\beta}\nonumber\\
&=& r/12 + rK^{-\gamma/\beta}\dfrac{1}{1-(1-\delta^2)^{\gamma/\beta}}. 
\end{eqnarray}
Choose 
$$ K = \eta\vee(1+\delta^{-1})\vee
12^{\beta/\gamma}(1-(1-\delta^2)^{\gamma/\beta})^{-\beta/\gamma}.$$
Then $K^{-\gamma/\beta}(1-(1-\delta^2)^{\gamma/\beta})^{-1}\leq 1/12,$
and hence $x_k\in B\cap B(Q, r/6)$ for every $k\geq 1.$ Since
$\lim_{k\to\infty}u(x_k) = +\infty,$ this contradicts the fact that
$u$ is bounded on $B(Q,r/2).$ This proves that $u(x) < K$ for every
$x\in B\cap B(Q,r/12)$ and completes the proof of the theorem.
\end{proof}
We now provide the proof of the lemma.

\vspace{0.1in}

\noindent \textbf{Proof of Lemma \ref{4lemma}:}
Let $\eta_0 = 2^{-2(d+\alpha)/d}.$ Then for $\rho_Q(x) < \eta_0 r,$
$$ \left(\rho_Q(x)\right)^{d/(\alpha + d)} < 1/4, \hspace{0.05in}\mbox{and}\hspace{0.05in} 2\rho_Q(x)\leq r^{1-\gamma}\rho_Q(x)^{\gamma} - 2\rho_Q(x).$$
Thus if $x\in B\cap B(Q, r/12)$ with $\rho_Q(x) < \eta_0 r,$ then $\mid x-y\mid  < 2\mid z-y\mid $ for $z\in B_0,$ $y\notin B_1.$ Now using the Levy system formula from equation \eqref{levy}, and the upper bound on expected exit time from \eqref{foonL}, we have 
\begin{eqnarray}\label{5}
\lefteqn{\mathbb{E}^x\left(u(X_{\tau_0}); X_{\tau_0}\notin B_1\right)= c_1\int_{B_0}G_{B_0}(x, z)\int_{\mid y-x\mid  >  r^{1-\gamma}\rho_Q(x)^{\gamma}}J(z, y) u(y)dydz} \nonumber\\ &\leq&
c_2\int_{B_0}G_{B_0}(x, z)dz\int_{\mid y-x\mid  >
  r^{1-\gamma}\rho_Q(x)^{\gamma}}\mid z-y\mid^{-d-\alpha}u(y)dy
\nonumber\\ &\leq& 2^{d + \alpha}c_3\int_{B_0}G_{B_0}(x, z)dz\int_{\mid y-x\mid 
  > r^{1-\gamma}\rho_Q(x)^{\gamma}}\mid x-y\mid^{-d-\alpha}u(y)dydz
\nonumber\\ &\leq& c_4 \mathbb{E}^x\left(\tau_{B(x,
    2\rho_Q(x))}\right)\int_{\mid y-x\mid  >
  r^{1-\gamma}\rho_Q(x)^{\gamma}}\mid x-y\mid^{-d-\alpha}u(y)dy
\nonumber\\
&\leq&c_5\rho_Q(x)^2(I_1+ I_2),
\end{eqnarray}
with 
$$  I_1 = \int_{\mid y-x\mid  >
  r^{1-\gamma}\rho_Q(x)^{\gamma},
  \mid y-x_0\mid >2\rho_Q(x_0)/3}\mid x-y\mid^{-d-\alpha}u(y)dy$$  and  $$I_2 \int_{\mid y-x\mid \leq
  2\rho_Q(x_0)/3}\mid x-y\mid^{-d-\alpha}u(y)dy.$$ 
On the other hand, for $z\in B_2,$ and $y\notin B_3,$ we have $$\mid
z-y\mid \leq \mid z-x_0\mid + \mid x_0-y\mid \leq\rho_Q(x_0)/3 + \mid
x_0-y\mid \leq 2\mid x_0-y\mid .$$ Now we apply the Levy system formula
and the bound in equation \eqref{foonLL} to obtain
\begin{eqnarray}\label{6}
u(x_0)&\ge& \mathbb{E}^x\left(u(X_{\tau_2}), X_{\tau_2}\notin
  B_3\right) \nonumber\\ &\ge& c_6\int_{B_2}G_{B_2}(x_0,
z)\int_{\mid y-x_0\mid >2\rho_Q(x_0)/3}\mid z-y\mid^{-d-\alpha}u(y)dydz
\nonumber\\ &\ge& 2^{-d-\alpha}c_7\int_{B_2}G_{B_2}(x_0,
z)dz\int_{\mid y-x_0\mid >2\rho_Q(x_0)/3}\mid x_0-y\mid^{-d-\alpha}u(y)dy
\nonumber\\ &\ge&
2^{-d-\alpha}c_8(\rho_Q(x_0)/3)^2\int_{\mid y-x_0\mid >2\rho_Q(x_0)/3}\mid x_0-y\mid^{-d-\alpha}u(y)dy
\nonumber\\ &=&
c_9\rho_Q(x_0)^2\int_{\mid y-x_0\mid >2\rho_Q(x_0)/3}\mid x_0-y\mid^{-d-\alpha}u(y)dy.
\end{eqnarray}
We shall use (\ref{6}) to estimate $I_1$. Now suppose that $\mid y-x\mid  > r^{1-\gamma}\rho_Q(x)^{\gamma}$ and $x\in B(Q, r/4).$ Then, 
$$\mid y-x_0\mid \leq \mid y-x\mid  + r\leq \mid y-x\mid  + r^{\gamma}\rho_Q(x)^{-\gamma}\mid y-x\mid \leq 2r^{\gamma}\rho_Q(x)^{-\gamma}\mid y-x\mid .$$
Therefore, 
\begin{eqnarray}\label{7}
I_1&=& \int_{\mid y-x\mid  > r^{1-\gamma}\rho_Q(x)^{\gamma},
  \mid y-x_0\mid >2\rho_Q(x_0)/3}\mid x-y\mid^{-d-\alpha}u(y)dy \nonumber\\ &\le&
\int_{\mid y-x_0\mid >2\rho_Q(x_0)/3}(2^{-1}(\rho_Q(x)/r)^{\gamma})^{-d-\alpha}\mid y-x_0\mid^{-d-\alpha}u(y)dy
\nonumber\\ &=&
2^{d+\alpha}(\rho_Q(x)/r)^{-\gamma(d+\alpha)}\int_{\mid y-x_0\mid >2\rho_Q(x_0)/3}\mid y-x_0\mid^{-d-\alpha}u(y)dy\nonumber\\ &\le&c_{10}
2^{d+\alpha}(\rho_Q(x)/r)^{-\gamma(d+\alpha)}\rho_Q(x_0)^{-2}u(x_0)
\nonumber\\ &=&
c_{11}(\rho_Q(x)/r)^{-\gamma(d+\alpha)}\rho_Q(x_0)^{-2}u(x_0),
\end{eqnarray}
where the last inequality above is due to \eqref{6}. If $\mid y-x_0\mid  < 2\rho_Q(x_0)/3,$ then $$\mid y-x\mid \geq \mid x_0-Q\mid -\mid x-Q\mid -\mid y-x_0\mid  > \rho_Q(x_0)/6.$$ This combined with the Harnack inequality (Proposition \ref{hq0}) gives
\begin{eqnarray}\label{8}
I_2&=& \int_{\mid y-x\mid \leq 2\rho_Q(x_0)/3}\mid x-y\mid^{-d-\alpha}u(y)dy \nonumber\\
   &\le& c_{12}\int_{\mid y-x\mid \leq 2\rho_Q(x_0)/3}\mid x-y\mid^{-d-\alpha}u(x_0)dy \nonumber\\
   &\le& c_{13}u(x_0)\int_{\mid y-x\mid \geq 2\rho_Q(x_0)/6}\mid x-y\mid^{-d-\alpha}dy \nonumber\\
   &=& c_{14}(\rho_Q(x_0))^{-\alpha}u(x_0).
\end{eqnarray}
Combining \eqref{5},\eqref{7}, and \eqref{8} we obtain 
\begin{eqnarray}\label{9}
\lefteqn{\mathbb{E}^x\left(u(X_{\tau_0}); X_{\tau_0}\notin B_1\right)}\nonumber\\
&\le&
  c_{15}\rho_Q(x)^2\left(c_{16}(\rho_Q(x)/r)^{-\gamma(d+\alpha)}\rho_Q(x_0)^{-2}u(x_0)
  + c_{17}(\rho_Q(x_0))^{-\alpha}u(x_0)\right)\nonumber\\ &\le&
  c_{18}u(x_0)\left((\rho_Q(x)/r)^{2-\gamma(d+\alpha)}+
  \rho_Q(x)^2r^{-\alpha}\right), \nonumber\\
\end{eqnarray}
where in the last inequality we used the fact that $\rho_Q(x_0) = r/2.$ Now we choose $\eta\in (0, \eta_0)$ so that  $$c_{18}\left(\eta^{2-\gamma(d-\alpha)} + \eta^2\right)\leq 1.$$
Then for $x\in B\cap B(Q, r/12)$ with $\rho_Q(x) < \eta r,$ we have by \eqref{9}
\begin{eqnarray*}
\mathbb{E}^x\left(u(X_{\tau_0}); X_{\tau_0}\notin B_1\right)&\le&
c_{18}u(x_0)\left(\eta^{2-\gamma(d-\alpha)} + \eta^2 r^{2-\alpha}\right)
\nonumber\\ &\le& c_{18}\left(\eta^{2-\gamma(d-\alpha)} + \eta^2
\right)u(x_0)\leq u(x_0),
\end{eqnarray*}
which proves the result.
\qed

\subsection{Boundary Harnack Principle}
We shall first provide estimates for exit probabilities near the
boundary. For this we shall consider the same for a truncated
process. We define the truncated process $\widehat X$ to be the
process with the same diffusion component as $X$ but with the jump
kernel to be $J(x,y)1_{\{\mid x-y\mid < 1\}}.$ That is the jump sizes
are restricted to be strictly smaller than $1.$ The corresponding exit times
will be denoted by $\widehat\tau.$
\begin{lemma}\label{bdryest}
There exist positive constants $\delta_0 = \delta_0(R_1, M_1, \alpha),$ $c_1 = c_1(R_1, M_1, \alpha)$
and, $c_2 = c_2(R_1, M_1, \alpha)$ such that for every $Q\in\partial B,$
and $x\in D_Q(2\delta_0, r_0)$ with $\tilde x = 0,$
\begin{equation}\label{lbd}
\mathbb{P}^x\left(\widehat X_{\widehat\tau_{D_Q(\delta_0, r_0)}}\in D_Q(2\delta_0, r_0)\right)\geq c_1\delta_B(x), \mbox{ and}
\end{equation}
\begin{equation}\label{ubd}
\mathbb{P}^x\left(\widehat X_{\widehat\tau_{D_Q(\delta_0, r_0)}}\in B\right)\leq c_2\delta_B(x). 
\end{equation}
\end{lemma}
\begin{proof}
Without loss of generality, assume that $Q = 0.$ Let $p>0$ be such that $p\neq\alpha$ and $1 < p < (2\wedge 3-\alpha).$ Recall from Notation \ref{not1}, that for $y \in B$, $\rho(y) = y_d - \phi(\tilde y)$ and $D(r_1, r_2) = \{y\in B: 0 < \rho(y) < r_1, \mid \tilde y\mid  < r_2\}.$ Define for $y \in B$,
$$h(y) = \rho(y)1_{B(0, R_0)\cap B},\hspace{0.1in} h_p(y) = h(y)^p,$$  
Since $\rho(y)\leq\sqrt{1+ M_1^2}\delta_B(y),$ we have
$0\leq h(y)\leq 1.$ Also observe that $D(r_1, r_2)$ is contained in
$B\cap B(0, R_1/4)$ for every $r_1, r_2\leq r_0.$ Let $\widehat L_d$ denote the integral term in the operator $\mathcal L$ but with the truncated kernel $J(x, y)1_{\{|x-y|<1\}}.$ Let $$L_cf(x)
=\frac12 \sum_{i=1}^d \sum_{j=1}^d \frac{\partial}{\partial x_i}
\left(a_{ij}(x) \frac{\partial f}{\partial x_j}\right)(x) $$ denote
the diffusion part of the operator $\mathfrak L.$ Applying the product
rule in $L_c,$ we will get two kinds of terms, one involving first
order derivatives of $a_{ij}$ and $f$ and the other involving $a_{ij}$
and second order derivatives of $f.$ We denotes these by $L_c^1$ and
$L_c^2$ respectively. For every $y\in B(0, R_0)\cap B,$ we have
\begin{equation}\label{L1a}
L_ch(y) = -L_c\phi(\tilde y),
\end{equation}
Next, a routine but tedious computation gives $L_c^2h_p(y) = I - II,$ where  
\begin{eqnarray}\label{L2}
I &=& p(p-1)h^{p-2}(y)\big[1+ \sum_{i=1}^{d-1}a_{ii}(y)(\partial_{i}\phi(\tilde y))^2\nonumber \\ &&
 + \sum_{1\leq i, j\leq d-1, i\neq j}a_{ij}(y)\partial_{i}\phi(\tilde y)\partial_{j}\phi(\tilde y) - 2\sum_{i=1}^{d-1}a_{id}(y)\partial_{i}\phi(\tilde y)\big]\nonumber\\
\end{eqnarray}
and
\begin{eqnarray}\label{L3}
II &=& ph^{p-1}(y)\big[\sum_{i=1}^{d-1}a_{ii}(y)\partial_{ii}\phi(\tilde y) + \sum_{1\leq i, j\leq d-1, i\neq j}a_{ij}(y)\partial_{ij}\phi(\tilde y)\big]
\end{eqnarray}
Using Assumption (A), the term
inside the square brackets in $I$ can be bounded from below by
$\lambda(1 + \mid \nabla\phi(\tilde y)\mid^2).$ Also observe that the term inside
the square brackets in $II$ is just $L_c^2\phi(\tilde y),$ so
that $II = ph^{p-1}(y)L_c^2\phi(\tilde y).$ Combining the previous two observations, we obtain
\begin{equation}\label{L4}
L_c^2h_p(y)\geq\lambda p(p-1)h^{p-2}(y)(1 + \mid \nabla\phi(\tilde y)\mid^2) - ph^{p-1}(y)L_c^2\phi(\tilde y).
\end{equation}
Now we take care of the first order term $L_c^1.$ Another routine computation gives 
$$L_c^1h_p(y) = -ph^{p-1}(y)\left[\sum_{1\leq i, j\leq d-1}(\partial_{i}a_{ij}(y))(\partial_{j}\phi(\tilde y))   -\sum_{i=1}^{d}\partial_{i}a_{id}(y) + \sum_{j=1}^{d}\partial_{d}a_{dj}(y)\partial_{j}\phi(\tilde y)\right].$$
 The first term in the square bracket above is nothing but
$L_c^1\phi(\tilde y),$ and also note that the second and third terms are
bounded because $a_{ij}$ and $\phi$ have bounded first
derivatives. So we may now write
\begin{equation}\label{L5}
L_c^1h_p(y) = -ph^{p-1}(y)L_c^1\phi(\tilde y) + ph^{p-1}(y)O(1)
\end{equation}
Adding equations \eqref{L4} and \eqref{L5}, we get
\begin{eqnarray}\label{L6}
L_ch_p(y) &=& L_c^2h_p(y) + L_c^1h_p(y)\nonumber\\
          &\ge & \lambda p(p-1)h^{p-2}(y)(1 + (\nabla\phi(\tilde y))^2)\\ &&
           - ph^{p-1}(y)L_c\phi(\tilde y)+ ph^{p-1}(y)O(1).
\end{eqnarray}
It now follows that $\delta_1$ maybe chosen sufficiently small that 
\begin{equation}\label{L7}
L_ch_p(y)\geq c_1(\rho(y))^{p-2} > 0
\end{equation}
for $y\in D(\delta_1, r_0)$ and appropriate constant $c_1 >0.$
We will divide the rest of the proof into three steps.

{\bf Step 1.} \emph{Constructing suitable superharmonic and
  subharmonic functions with respect to $L_c + \widehat L_d.$ } Let
$\psi$ be a smooth positive function in $\mathbb{R}^d$ with bounded
first and second order partial derivatives such that $\psi(y) =
2^{p+1}\frac{\mid \tilde y\mid^2}{r_0^2}$ for $\mid y\mid  < r_0/4$ and
$2^{p+1}\leq\psi(y)\leq 2^{p+2},$ for $\mid y\mid  > r_0/2.$ We now define
$$ u_1(y) = h(y) + h_p(y)$$
and
$$ u_2(y) = h(y) + \psi(y) - h_p(y)$$
Note that both $u_1$ and $u_2$ are non-negative because $0\leq h\leq 1$ and $p\geq 1.$ By a Taylor expansion with remainder of order $2,$ 
\begin{equation}\label{L8}
\mid (L_c + \widehat L_d)\psi(y)\mid \leq \mid L_c\psi(y)\mid  + \mid \widehat L_d\psi(y)\mid \leq c_2(\alpha) < \infty.
\end{equation}
Note that our jump kernel $J(x, y)$ is uniformly bounded above and
below by upto a constant by $\mid x-y\mid^{d+\alpha},$ so we have by [\cite{cskv} Corollary
  $3.3$], that there exist $c_3 = c_3(R_1, M_1)$ and $\delta_2 =
\delta_2(R_1, M_1)\in (0, \delta_1]$ such that
$$\widehat L_dh_p(y)\geq -c_3\hspace{0.1in}\mbox{for}\hspace{0.05in}y\in D(\delta_2, r_0).$$
Then using (\ref{L7}), the fact that $p<2,$ and the above inequality, we obtain (choosing $\delta_2$ smaller if need be)
\begin{equation}\label{L9}
(L_c + \widehat L_d)h_p(y)\geq c_1\rho(y)^{p-2} - c_3\geq \frac{c_1}{2}\rho(y)^{p-2},
\end{equation}
for $y\in D(\delta_2, r_0).$ Making use of [\cite{cskv} Corollary
  $3.3$] again and (\ref{L1a}), there exist $c_4$ and $\delta_3\in (0,
\delta_2)$ such that for all $y\in D(\delta_3, r_0)$
\begin{equation}\label{L10}
\mid (L_c + \widehat L_d)h(y)\mid \leq c_4\big(1 + \rho(y)^{(1-\alpha)\wedge 0} + 1_{\{\alpha =1\}}\mid \log\rho(y)\mid \big)
\end{equation}
Thus by \eqref{L8}-\eqref{L10} and the fact that $p < 2\wedge(3-\alpha),$ there exists $\delta_4\in (0, \delta_3)$ such that 
\begin{equation}\label{L11}
(L_c + \widehat L_d)u_2(y)\leq c_2 + c_4\big(1 + \mid \log\rho(y)\mid  + \rho(y)^{(1-\alpha)\wedge 0}\big) - \frac{c_1}{2}\rho(y)^{p-2}\leq -1
\end{equation}
for $y\in D(\delta_4, r_0).$
On the other hand, the lower bound from [\cite{cskv} Corollary $3.3$] gives
\begin{equation}\label{L12}
(L_c + \widehat L_d)h(y)\geq -\mid \mid L_c\phi\mid \mid _{\infty} - c_5(1 + \mid \log\rho(y)\mid )
\end{equation}
for all $y\in D(\delta_4, r_0).$ Combining \eqref{L12} with
\eqref{L9} and choosing $\delta_4$ smaller if necessary, we obtain
that for $y\in D(\delta_4, r_0),$
\begin{equation}\label{L13}
(L_c + \widehat L_d)u_1(y)\geq -\mid \mid L_c\phi\mid \mid _{\infty} - c_5(1 + \mid \log\rho(y)\mid )+ \frac{c_1}{2}\rho(y)^{p-2}\geq 0.
\end{equation}

{\bf Step 2.} \emph{From sub/super- harmonic functions to
  sub/super-martingale properties} We claim that the estimates
\eqref{L11} and \eqref{L13} imply that
\begin{equation}\label{L14}
t\rightarrow u_2(\widehat X_{t\wedge\widehat\tau_{D(\delta_4, r_0)}})\hspace{0.1in}\mbox{is a bounded supermartingale,}\hspace{0.1in}
\end{equation}
\begin{equation}\label{L15}
\mathbb{E}^x\left(\widehat\tau_{D(\delta_4, r_0)} \right)\leq\rho(x),
\end{equation}
and 
\begin{equation}\label{L16}
t\rightarrow u_1(\widehat X_{t\wedge\widehat\tau_{D(\delta_4, r_0)}})\hspace{0.1in}\mbox{is a bounded submartingale.}\hspace{0.1in}
\end{equation}
Note that if $v$ is a bounded $C^2$ function in $\mathbb{R}^d$ with
bounded first and second order partial derivatives, then an
application of Ito's formula and the Levy system \eqref{levy} gives
\begin{equation}\label{L17}
M_t^v = v(\widehat X_t) - v(\widehat X_0) - \int_{0}^{t}(L_c + \widehat L_d)v(\widehat X_s)ds
\end{equation}
is a martingale. If the functions $u_1$ and $u_2$ were $C^2$ with
bounded derivatives, then the above claims would just follow from
(\ref{L17}), (\ref{L11}), and (\ref{L13}). Since they are truncated outside of $B(0, R_0)\cap B$ the functions are not in $C^2$. So we will proceed by using a mollifier. Let $g$ be a non-negative
smooth radial function with compact support in $\mathbb{R}^d$ such
that $g(x) = 0$ for $\mid x\mid  > 1,$ and $\int_{\mathbb{R}^d}g(x)dx =1.$ For
$k\geq 1,$ denote $g_k(x) = 2^{kd}g(2^kx).$ Define for $i = 1, 2,$
$$u_i^k(z) = g_k*u_i(z):= \int_{\mathbb{R}^d}g_k(y)u_i(z-y)dy.$$ Since
$(L_c + \widehat L_d)u_i^k = g_k*(L_c + \widehat L_d)u_i,$ we have by
(\ref{L11}) and (\ref{L13}) that $(L_c + \widehat L_d)u_1^k\geq 0$ and
$(L_c + \widehat L_d)u_2^k\leq -1,$ on $D_k(\delta_4, r_0) = \{y:
2^{-k} < \rho(y) < \delta_4 - 2^{-k}, \mid \tilde y\mid  < r_0 - 2^{-k}\}.$
Since $u_i^k$ $i= 1,2$ are bounded smooth functions with bounded first
and second order partials bounded, equation (\ref{L17}) tells us that
$$t\rightarrow u_2^k(\widehat X_{t\wedge\widehat\tau_{D_k(\delta_4, r_0)}}) + t\wedge\widehat\tau_{D_k(\delta_4, r_0)}$$
is a positive supermartingale and similarly that 
$$t\rightarrow u_1^k(\widehat X_{t\wedge\widehat\tau_{D_k(\delta_4, r_0)}})$$
is a bounded submartingale. Since $u_i$ are bounded and continuous, $u_i^k$ converge uniformly to $u_i.$ Thus 
\begin{equation}\label{L18}
t\rightarrow u_2(\widehat X_{t\wedge\widehat\tau_{D_k(\delta_4, r_0)}}) + t\wedge\widehat\tau_{D_k(\delta_4, r_0)} \hspace{0.1in}\mbox{is a positive supermartingale}\hspace{0.05in}
\end{equation}
and $t\rightarrow u_1(\widehat X_{t\wedge\widehat\tau_{D_k(\delta_4,
    r_0)}})$ is a bounded submartingale. Since $D_k(\delta_4, r_0)$
increases to $D(\delta_4, r_0)$ we see that (\ref{L14}) and (\ref{L16})
hold. In addition, for each fixed $k\geq 1,$ and $t>0,$ we have from
(\ref{L18}) that
$$\mathbb{E}^x\left(u_2(\widehat X_{t\wedge\widehat\tau_{D_k(\delta_4, r_0)}}) + t\wedge\widehat\tau_{D_k(\delta_4, r_0)} \right)\leq u_2(x) $$
Since $u_2\geq 0,$ by first letting $k\to\infty$ and then $t\to\infty,$ we get $\mathbb{E}^x\left(\widehat\tau_{D(\delta_4, r_0)} \right)\leq u_2(x).$ Since $\tilde x =0, \psi(x) = 0$ and therefore $u_2(x)\leq\rho(x).$ This proves (\ref{L15}).

{\bf Step 3.} \emph{Deriving exit distribution estimates using
  sub/super-martingale property} Recall that $\psi\geq 2^{p+1}$ on
$\mid y\mid \geq r_0$ and $\psi(x) = 0.$ Therefore by \eqref{L14},
\begin{eqnarray}\label{L19}
\rho(x) &\ge & u_2(x) \nonumber\\
        &\ge & \mathbb{E}^x\left(u_2(\widehat X_{\widehat\tau_{D(\delta_4, r_0)}}); X_{\widehat\tau_{D(\delta_4, r_0)}}\in B\setminus{D(\infty, r_0)} \right) \nonumber\\
        &\ge & (2^{p+1}-1)\mathbb{P}^x\left(X_{\widehat\tau_{D(\delta_4, r_0)}}\in B\setminus{D(\infty, r_0)} \right)
\end{eqnarray}
From (\ref{L16}), we have
\begin{eqnarray}\label{L20}
\rho(x)\leq\rho(x) + \rho(x)^p &= & u_1(x)\nonumber\\
&\le & \mathbb{E}^x\left(u_1(\widehat X_{\widehat\tau_{D(\delta_4, r_0)}}) \right)\nonumber\\
&\le & 2\mathbb{P}^x\left(\widehat X_{\widehat\tau_{D(\delta_4, r_0)}}\in B\right).
\end{eqnarray}
Combining (\ref{L19}) and (\ref{L20}), we obtain 
\begin{eqnarray}\label{L21}
\mathbb{P}^x\left(\widehat X_{\widehat\tau_{D(\delta_4, r_0)}}\in
D(\infty, r_0)\right)& = & \mathbb{P}^x\left(\widehat
X_{\widehat\tau_{D(\delta_4, r_0)}}\in B\right) -
\mathbb{P}^x\left(X_{\widehat\tau_{D(\delta_4, r_0)}}\in
B\setminus{D(\infty, r_0)} \right)\nonumber\\ &\ge & \frac{2^{p+1} -
  3}{2(2^{p+1}-1)}\rho(x).
\end{eqnarray}
Now we use the Levy system (\ref{levy}) for $\widehat X$ to get 
\begin{eqnarray}\label{L22}
\lefteqn{\mathbb{P}^x\left(\widehat X_{\widehat\tau_{D(\delta_4, r_0)}}\in
D(\infty, r_0)\setminus D(2\delta_4, r_0)\right) }\nonumber\\ &\le
& c\mathbb{E}^x\left(\int_{0}^{\widehat\tau_{D(\delta_4, r_0)}}
  \int_{D(\infty, r_0)\setminus D(2\delta_4, r_0)
  }\frac{1_{\{\mid \widehat X_s - y\mid < 1\}}}{\mid \widehat X_s -
    y\mid^{d+\alpha}}dyds\right) \nonumber\\ &\le &
c\mathbb{E}^x\left(\int_{0}^{\widehat\tau_{D(\delta_4, r_0)}}
  \int_{D(\infty, r_0)\setminus D(2\delta_4, r_0) }\frac{1}{\mid \widehat
    X_s - y\mid^{d+\alpha}}dyds\right) \nonumber\\ &\le &
c_6\left(\int_{D(\infty, r_0)\setminus D(2\delta_4,
  r_0)}\mid y\mid^{-d-\alpha}dy
\right)\mathbb{E}^x\left(\widehat\tau_{D(\delta_4, r_0)}\right)
\nonumber\\ &\le & c_7\left(\int_{D(2\delta_4, r_0)\setminus
  D(3\delta_4/2, r_0)}\mid y\mid^{-d-\alpha}dy
\right)\mathbb{E}^x\left(\widehat\tau_{D(\delta_4, r_0)}\right)
\nonumber\\ &\le &
c_8\mathbb{E}^x\left(\int_{0}^{\widehat\tau_{D(\delta_4,
      r_0)}}\int_{D(2\delta_4, r_0)\setminus D(3\delta_4/2, r_0)
  }\frac{1_{\{\mid \widehat X_s - y\mid < 1\}}}{\mid \widehat X_s -
    y\mid^{d+\alpha}}dyds\right) \nonumber\\ & = &
c_9\mathbb{P}^x\left(\widehat X_{\widehat\tau_{D(\delta_4, r_0)}}\in
D(2\delta_4, r_0)\setminus D(3\delta_4/2, r_0) \right) .
\end{eqnarray}
Thus from (\ref{L21}) and (\ref{L22}), 
\begin{equation}\label{L23}
\mathbb{P}^x\left(\widehat X_{\widehat\tau_{D(\delta_4, r_0)}}\in D(2\delta_4, r_0)\right)\geq c_{10}\rho(x)\geq c_{11}\delta_B(x).
\end{equation}
Taking $\delta_0 = \delta_4$ and $c_1 = c_{11}$ gives the estimate \eqref{lbd}. To obtain \eqref{ubd}, first recall that $0\leq h_p\leq 1.$ If $|y| > r_0/2,$ then $\psi(y)\geq 2^{p+1}\geq 1.$ Therefore 
$$ u_2(y)= \psi(y) + h(y) - h_p(y)\geq \psi(y) - h_p(y)\geq 1, \hspace{0.05in}\mbox{for}\hspace{0.05in} y\in B(0, r_0/2)^c.$$
Note also that for $y\in B(0, R_0)$ such that $\delta_4\leq\rho(y) < R_0,$ 
$$u_2(y)= \psi(y) + h(y) - h_p(y)\geq\rho(y) - \rho(y)^p\geq c_{12},$$
where $c_{12}$ depends on $\delta_4$ and $R.$ Therefore using the subharmonicity, we obtain
\begin{equation}\label{ubdfinal}
\rho(x)\geq u_2(x)\geq\mathbb{E}^x\left(u_2(\widehat X_{\widehat\tau_{D(\delta_4, r_0)}}) \right)\geq c_{12}\mathbb{P}^x\left(\widehat X_{\widehat\tau_{D(\delta_4, r_0)}}\in B \right)
\end{equation}
Since $\rho(x)$ is comparable to $\delta_B(x)$ from above and below, we infer from \eqref{ubdfinal} that \eqref{ubd} is true (once again with $\delta_0 = \delta_4$). This finishes the proof of the lemma.
\end{proof}
We now state and prove the proposition that provides estimates for
exit probabilities near the boundary.
\begin{proposition}\label{bdryest2} There exist positive constants $\delta_0 = \delta_0(R_1, M_1, \alpha),$
 $c_1 \equiv c_1(R_1, M_1, \alpha)$ and $c_2 \equiv c_2(R_1, M_1, \alpha)$ such that for every $Q\in\partial B,$
and $x\in D_Q(2\delta_0, r_0)$ with $\tilde x = 0,$
\begin{eqnarray}\label{propl}
\mathbb{P}^x\left(X_{\tau_{D(\delta_0, r_0)}}\in D(2\delta_0, r_0)\right)&\geq& c_1\delta_B(x),\\
\label{propu}\mathbb{P}^x\left(X_{\tau_{D(\delta_0, r_0)}}\in B\right)&\leq& c_2\delta_B(x).
\end{eqnarray}
\end{proposition}
\begin{proof}
The estimate \eqref{propl} follows from Lemma \ref{bdryest} and the fact that 
\begin{equation}
\mathbb{P}^x\left(X_{\tau_{D(\delta_0, r_0)}}\in D(2\delta_0,
r_0)\right)= \mathbb{P}^x\left(\widehat X_{\widehat\tau_{D(\delta_0,
    r_0)}}\in D(2\delta_0, r_0)\right)\geq c_{11}\delta_B(x).
\end{equation}
To get \eqref{propu} we will once again use Lemma \ref{bdryest}. From that lemma, we know that $$\mathbb{P}^x\left(\widehat X_{\widehat\tau_{D(\delta_0, r_0)}}\in B \right)\leq c\delta_B(x).$$
We would like to obtain a similar bound for the process $X.$ Using the Levy system formula, one has
\begin{eqnarray}\label{propu1}
\mathbb{P}^x\left(X_{\tau_{D(\delta_0, r_0)}}\in B\setminus D(2\delta_0, 2r_0) \right)& = & 
\mathbb{E}^x\left[\int_{0}^{\tau_{D(\delta_0, r_0)}}\int_{B\setminus D(2\delta_0, 2r_0)}J(X_s, y) dy ds\right]\nonumber\\
&\le & c_1 \mathbb{E}^x\left[\int_{0}^{\tau_{D(\delta_0, r_0)}}\int_{B\setminus D(2\delta_0, 2r_0)}\frac{1}{|X_s - y|^{d + \alpha}}dy ds\right] \nonumber\\
& \le & c_2\left(\int_{B\setminus D(2\delta_0, 2r_0)}|y|^{-d -\alpha}dy\right)\mathbb{E}^x(\tau_{D(\delta_0, r_0)})\nonumber\\
&\le& c_3\left(\int_{D(2\delta_0, r_0)\setminus D(3\delta_0/2, r_0)}|y|^{-d -\alpha}dy\right)\mathbb{E}^x(\tau_{D(\delta_0, r_0)})\nonumber\\
&\le& c_4\mathbb{E}^x\left[\int_{0}^{\tau_{D(\delta_0, r_0)}}\int_{D(2\delta_0, r_0)\setminus D(3\delta_0/2, r_0) }\frac{1}{|X_s - y|^{d + \alpha}}dy ds\right] \nonumber\\
&\le& c_5\mathbb{E}^x\left[\int_{0}^{\tau_{D(\delta_0, r_0)}}\int_{D(2\delta_0, r_0)\setminus D(3\delta_0/2, r_0) }J(X_s, y) dy ds\right]\nonumber\\
&=& c_5\mathbb{P}^x\left(\widehat X_{\widehat\tau_{D(\delta_0, r_0)}}\in D(2\delta_0, r_0)\setminus D(3\delta_0/2, r_0)\right)
\end{eqnarray}
Therefore we have that 
\begin{eqnarray}
\mathbb{P}^x\left(X_{\tau_{D(\delta_0, r_0)}}\in B\right)&=& \mathbb{P}^x\left(X_{\tau_{D(\delta_0, r_0)}}\in B\setminus D(2\delta_0, 2r_0)\right) + \mathbb{P}^x\left(\widehat X_{\widehat\tau_{D(\delta_0, r_0)}}\in D(2\delta_0, 2r_0)\right) \nonumber\\
&\le& c_5\mathbb{P}^x\left(X_{\tau_{D(\delta_0, r_0)}}\in D(2\delta_0, r_0)\setminus D(3\delta_0/2, r_0)\right) + \mathbb{P}^x\left(\widehat X_{\widehat\tau_{D(\delta_0, r_0)}}\in B\right) \nonumber\\
&\le& c_5\mathbb{P}^x\left(\widehat X_{\widehat\tau_{D(\delta_0, r_0)}}\in D(2\delta_0, r_0)\setminus D(3\delta_0/2, r_0)\right) + \mathbb{P}^x\left(\widehat X_{\widehat\tau_{D(\delta_0, r_0)}}\in B\right) \nonumber\\
&\le& (c_5 +1)\mathbb{P}^x\left(\widehat X_{\widehat\tau_{D(\delta_0, r_0)}}\in B\right) \nonumber\\
&\le& c_6\delta_B(x)
\end{eqnarray}
In the second line above, we have used \eqref{propu1} and  in the last but one line we have used \eqref{ubd}. This finishes the proof of the proposition. 
\end{proof}
We now have all the ingredients to state and prove the following Uniform Boundary Harnack Principle (BHP).
\begin{theorem}\label{BdryHarnack}
Let $B$ be a fixed ball with characteristics $R_1,$ and $M_1$ as above. There exists a positive constant $c_1\equiv c_1(\alpha, d, R_1, M_1)$ such that for $Q\in\partial B,$  $r \in (0, R_1),$ and any non-negative function $u$ on $\mathbb{R}^d,$ that is $\mathcal L$-harmonic in $B\cap B(Q, r),$  and vanishing continuously on $B^c\cap B(Q, r),$ we have
$$\dfrac{u(x)}{\delta_B(x)}\leq c_1\dfrac{u(y)}{\delta_B(y)}, \quad\mbox{for every}\hspace{0.05in} x, y\in B\cap B(Q, r/2).$$
\end{theorem}
\begin{proof}
By the Harnack principle and a chain argument it is sufficient to prove the inequality 
for $x, y \in B\cap B(Q, rr_0/8).$ We recall that $r_0 = \frac{R_1}{4(1 + M_1^2)}.$ For any $r\in (0,R_1]$ and $y\in B\cap B(Q, rr_0/8),$ let $Q_y$ be
  the point so that $\mid y - Q_y\mid  = \delta_B(y)$ and let $y_0 = Q_y +
  \frac{r}{8}\frac{(y- Q_y)}{\mid y- Q_y\mid }.$ Choose a smooth function
  $\phi:\mathbb{R}^{d-1}\rightarrow\mathbb{R}$ satisfying $\phi(0) =
  \nabla\phi(0) = 0,$ and $\mid  \nabla\phi(x) - \nabla\phi(y)\leq
  M_1\mid x-y\mid $ and an orthonormal coordinate system $\mbox{CS} \equiv CS_Q$ with its
  origins at $Q_y$ so that $B(Q_y, R_1)\cap B = \{y = (y_d, \tilde
  y)\in B(0, R_1)\hspace{0.05in}\mbox{in CS}\hspace{0.05in}: y_d >
  \phi(\tilde y)\}.$
In the above coordinate system $\tilde y = 0,$ and $y_0 = (0, r/8).$
For $a_1, a_2 >0$ define
$$D(a_1, a_2) = \{y\hspace{0.05in}\mbox{in CS}\hspace{0.05in}: 0 < y_d
- \phi(\tilde y) < a_1\frac{r\delta_0}{8}, \mid \tilde y\mid <
a_2\frac{rr_0}{8}\}.$$ Then it is easy to see that $D(2, 2)\subset B\cap
B(Q, r/2).$ Since $u$ is a harmonic function in $B\cap B(Q,
r)$ and vanishes continuously in $B^c\cap B(Q, r),$ it is regular
harmonic in $B\cap B(Q, r/2)$ and hence also in $D(2, 2)$ 
(c.f. Lemma $4.2$ \cite{cskv}). Thus by the Harnack inequality we have
\begin{eqnarray}\label{bhp1}
u(x)&=&\mathbb{E}^x\left(u(X_{\tau_{D(1, 1)}})\right) \geq\mathbb{E}^x\left(u(X_{\tau_{D(1, 1)}}); X_{\tau_{D(1, 1)}}\in D(2, 1)\right) \nonumber\\
&\ge& c_1u(x_0)\mathbb{P}^x\left(X_{\tau_{D(1, 1)}}\in D(2, 1)\right)\geq c_2u(x_0)\delta_B(x)/r,
\end{eqnarray}
where in the last line, we used \eqref{propl}. Let $w$ be the point with coordinates $(\tilde 0, rr_0/16).$ Then observe that there is a positive number $\eta\equiv\eta(M_1, \delta_0, r_0)\in (0, 1)$ such that $B(w, \eta rr_0/16)\in D(1,1).$ By the Levy system formula,
\begin{eqnarray}\label{bhp2}
u(w)&\ge& \mathbb{E}^w\left(u(X_{\tau_{D(1, 1)}}); X_{\tau_{D(1, 1)}}\notin D(2, 2)\right)\nonumber\\
&=&\int_{D(1, 1)} G_{D(1, 1)}(w, z)\int_{\mathbb{R}^d\setminus D(2, 2)}u(y)J(y, z)dy dz \nonumber\\
&\ge& c_3\mathbb{E}^w\left(\tau_{B(w, \eta rr_0/16)}\right)\int_{\mathbb{R}^d\setminus D(2, 2)}\frac{u(y)}{|w-y|^{d + \alpha}}dy \nonumber\\
&\ge& c_4 r^2 \int_{\mathbb{R}^d\setminus D(2, 2)}\frac{u(y)}{|w-y|^{d + \alpha}}dy.
\end{eqnarray}
Hence 
\begin{eqnarray}\label{bhp3}
\mathbb{E}^x\left(u(X_{\tau_{D(1, 1)}}); X_{\tau_{D(1, 1)}}\notin D(2, 2)\right)\nonumber\\
= \int_{D(1, 1)} G_{D(1, 1)}(x, z)\int_{\mathbb{R}^d\setminus D(2, 2)}u(y)J(y, z)dy dz \nonumber\\
\leq c_5\mathbb{E}^x\left(\tau_{D(1, 1)}\right)\int_{\mathbb{R}^d\setminus D(2, 2)}\frac{u(y)}{|w-y|^{d + \alpha}}dy \nonumber\\
\leq c_6\delta_B(x)r\int_{\mathbb{R}^d\setminus D(2, 2)}\frac{u(y)}{|w-y|^{d + \alpha}}dy\leq\dfrac{c_6\delta_B(x)}{c_4r}u(w).
\end{eqnarray}
On the other hand by the Harnack inequality and the Carleson estimate, we have
\begin{eqnarray}\label{bhp4}
\mathbb{E}^x\left(u(X_{\tau_{D(1, 1)}}); X_{\tau_{D(1, 1)}}\in D(2, 2)\right)&\le& c_7u(x_0)\mathbb{P}^x\left(X_{\tau_{D(1, 1)}}\in D(2, 2)\right)\nonumber\\
&\le& c_8 u(x_0)\delta_B(x)/r
\end{eqnarray}
where we used \eqref{propu}. Combining the two inequalities, we obtain
\begin{eqnarray}\label{bhp5}
u(x)&=& \mathbb{E}^x\left(u(X_{\tau_{D(1, 1)}}); X_{\tau_{D(1, 1)}}\notin D(2, 2)\right) + \mathbb{E}^x\left(u(X_{\tau_{D(1, 1)}}); X_{\tau_{D(1, 1)}}\in D(2, 2)\right) \nonumber\\
&\le& \dfrac{c_6\delta_B(x)}{c_4r}u(w) + \frac{c_8}{r} u(x_0)\delta_B(x) \nonumber\\
&\le& \frac{c_9}{r}\delta_B(x)\left(u(x_0) + u(w)\right) \nonumber\\
&\le& \frac{c_{10}}{r}\delta_B(x)u(x_0)
\end{eqnarray}
Thus by \eqref{bhp1} and \eqref{bhp5}, we observe that for every $x, y\in B\cap B(Q, rr_0/8),$ we have
$$ \dfrac{u(x)}{u(y)}\leq\dfrac{c_{10}}{c_2}\dfrac{\delta_B(x)}{\delta_B(y)},$$
which finishes the proof of the theorem.
\end{proof}
\section{Green Function Estimates} \label{GFE}
Our aim in this section is to prove a 3G-inequality (Proposition
\ref{3G}), using which we shall prove Proposition \ref{condgauge}.
We closely follow the idea of the proofs from \cite{cskv} and
\cite{cfz}. The basic ingredients in the proof of the the 3G estimate
are the Carleson estimate (see Theorem \ref{Carleson}) and the
Boundary Harnack Principle (Theorem \ref{BdryHarnack}).  

Recall that $J$ satisfies (\ref{dl2}). We would like to get upper
and lower bound estimates for the Green function $G_B(x,y)$ for the
ball $B.$ We use estimates on the transition density to get these bounds. Let $p(t,x,y)$ be the transition density for our process $X_t$. For $r >0$ let
$$p_c(t,r) = \dfrac{e^{\frac{-r^2}{t}}}{t^{d/2}} \mbox{ and } p_j(t, r) =
t^{\frac{-d}{\alpha}}\wedge\frac{t}{r^{d+\alpha}}.$$ 

\begin{lemma}\label{ck}
The transition density $p(t,x,y)$ satisfies 
$$p(t,x,y) \leq c_1(t^{\frac{-d}{\alpha}}\wedge t^{\frac{-d}{2}})\wedge (p_c(t, c_2\mid x-y\mid ) + p_j(t, \mid x-y\mid )),$$
and 
$$p(t,x,y) \geq c_3 t^{\frac{-d}{2}}, \mbox{ for } \mid x-y \mid^2 \leq t < 1$$
\end{lemma}
\begin{proof} See Theorem 1.4 in \cite{CK2}.
\end{proof}

\begin{lemma}\label{gbounds}
Let $G$ be the Green function for the process $X_t.$ Then,
\begin{enumerate}
\item[(a)]  for all
$x,y\in\mathbb{R}^d,$ there exist $c_1 >0$ such that 
\begin{equation} \label{upper} G(x,y)\leq \frac{c_1}{ \mid x-y\mid^{d-2}}.\end{equation}
\item[(b)] for all $x, y \in\mathbb{R}^d,$ with $\mid x-y\mid \leq 7R/8,$ there exist $c_2 >0$ such that
\begin{equation} \label{lower} G(x, y) \geq  \frac{c_2}{ \mid x-y\mid^{d-2}}.\end{equation}
\item[(c)] for all $x, y$  such that $\mid x-y\mid \leq 7R/8,$ there exists $L\geq 2$ and $c>0$ such that
\begin{equation} \label{lowerL}G(x, y) - G(Lx, Ly)\geq  \frac{c_3}{ \mid x-y\mid^{d-2}}.\end{equation}
\end{enumerate}
\end{lemma}
\begin{proof}
(a) By definition, \begin{eqnarray}
G(x,y)& = & \int_{0}^{\infty}p(t,x,y)dt \nonumber\\ 
      & = & \int_{0}^{\mid x-y\mid^2}p(t,x,y)dt + \int_{\mid x-y\mid^2}^{\infty}p(t,x,y)dt \nonumber\\
      & = & I_1 + I_2
\end{eqnarray}
Now, $$I_2 =
\int_{\mid x-y\mid^2}^{\infty}p(t,x,y)dt\leq c\int_{\mid x-y\mid^2}^{\infty}t^{\frac{-d}{2}}\leq\frac{c_1}{d-2}\mid x-y\mid^{2-d},$$
where we used the bound $p(t,x,y)\leq ct^{\frac{-d}{2}}$ from
Lemma \ref{ck}. To estimate $I_1$ we proceed along the following lines. Suppose first $\mid x-y\mid <1.$ Then using Lemma \ref{ck} again, 
\begin{eqnarray}\label{1}
I_1 & = & \int_{0}^{\mid x-y\mid^2}p(t,x,y)dt \nonumber\\
    &\le & \int_{0}^{\mid x-y\mid^2}p_c(t,x,y)dt + \int_{0}^{\mid x-y\mid^2}p_j(t,x,y)dt \nonumber\\
    &\le & c_2\mid x-y\mid^{2-d} + c_3\int_{0}^{\mid x-y\mid^2}\frac{t}{\mid x-y\mid^{d+\alpha}} \nonumber\\
    & = & c_2\mid x-y\mid^{2-d} + \frac{c_3}{2} \mid x-y\mid^{-d+4-\alpha}
\end{eqnarray}
In the third line above we have used the fact that $p_c$ is a Gaussian
density, the fact that $p_j(t, x, y) = \frac{t}{\mid x-y\mid^{d+\alpha}}$ for
$t\leq \mid x-y\mid^{\alpha}$ and that $\mid x-y\mid^2 < \mid x-y\mid^{\alpha}$ for $\mid x-y\mid 
< 1.$ Since $\alpha < 2$ and $\mid x-y\mid  < 1,$ the last line in the above
chain is bounded above by $c_4\mid x-y\mid^{2-d}.$ A similar estimate can be
proved for $I_1$ when $\mid x-y\mid  > 1.$ Combining the two estimates for
$I_1$ and $I_2$ proves (\ref{upper}).

(b) By Theorem $3.1$ in \cite{CK2}, one has $p(t, x, y)\geq
c_1t^{\frac{-d}{2}}, \mbox{ for } \mid x-y\mid^2\leq t\leq 1.$ Therefore
\begin{eqnarray*}
G(x,y)& = & \int_{0}^{\infty}p(t,x,y)dt \nonumber\\ &\ge &
\int_{\mid x-y\mid^2}^{1}c_1t^{\frac{-d}{2}}dt \nonumber\\ &\ge &
c_2\mid x-y\mid^{2-d}.
\end{eqnarray*}
(c) Using the bounds in (\ref{upper}) and (\ref{lower}), we can write 
$$G(x, y) - G(Lx, Ly)\geq (c_2 - \frac{c_1}{L^{d-2}})\mid x-y\mid^{2-d},$$
for $\mid x-y\mid \leq 7R/8.$ Now choose $L\geq 2$ such that $c_2 - \frac{c_1}{L^{d-2}}= c_3 >0.$ This finishes the proof.
\end{proof}
We are now ready to prove the required estimates for $G_B$. For any $x \in B,$ let $\delta_B(x) = \mbox{dist}(x, \partial B).$
\begin{lemma}\label{G_Bbd}
 Let $G_B$ denote the Green function for the killed process $X_t^B.$ Then,
\begin{equation} \label{ub1}G_B(x, y)\leq \frac{c_1}{\mid x -y \mid^{d-2}}\mbox{ for all } x, y\in B\end{equation}
and
\begin{equation} \label{lb1} G_B(x, y)\geq \frac{c_2}{\mid x -y \mid^{d-2}} \mbox{ when } 2\mid x-y\mid \leq\delta_B(x)\wedge\delta_B(y).\end{equation}
\end{lemma}
\begin{proof} Since
 $G_B(x, y)\leq G(x, y)$ the first part follows from (\ref{upper}). Without loss of generality assume that
  $\delta_B(y)\leq\delta_B(x).$ For the second part, we divide the
  proof into two cases. Let $\widetilde r_0 = R/8$ (any positive $\widetilde r_0$ strictly less than $R$ will also work) and let $L$ be such that Lemma \ref{gbounds} holds.

{\bf Case 1:} $L\mid x-y\mid \leq\delta_B(y)$ We consider three subcases.\\
 { (a)} $\delta_B(y)\leq \widetilde r_0.$ First observe that $B(y, \delta_B(y))\subset B.$ Then note that since
$L\mid x-y\mid \leq\delta_B(y),$ we have
$\mid X_{\tau_{B(y, \delta(y)}}- y\mid \geq\delta_B(y)\geq L\mid x-y\mid ,$ and
therefore
\begin{eqnarray}\label{GBlow}
G_B(x, y) & \ge & G_{B(y, \delta_B(y))}(x,y)\nonumber\\
          & = & G(x, y) - \mathbb{E}^x\left[G(X_{\tau_{B(y, \delta(y))}}, y)\right] \nonumber\\
          & \ge & G(x, y) - \frac{c_1}{L^{d-2}}\mid x-y\mid^{2-d} \nonumber\\
          & \ge & (c_2 - \frac{c_1}{L^{d-2}})\mid x-y\mid^{2-d}\nonumber\\
          & = & c_3\mid x-y\mid^{2-d}.
\end{eqnarray}
{ (b)} $\delta_B(y)> \widetilde r_0$ and $L \mid x - y\mid \leq\widetilde r_0.$ In this case, $\mid X_{\tau_{B(y, \widetilde r_0)}} - y\mid \geq\widetilde r_0\geq L \mid x - y\mid .$ Then,
\begin{eqnarray}\label{GBlow2}
G_B(x, y) & \ge & G_{B(y, \widetilde r_0)}(x,y)\nonumber\\
          & = & G(x, y) - \mathbb{E}^x\left[G(X_{\tau_{B(y, \widetilde r_0)}}, y)\right] \nonumber\\
          & \ge & G(x, y) - \frac{c_1}{L^{d-2}}\mid x-y\mid^{2-d} \nonumber\\
          & \ge & (c_2 - \frac{c_1}{L^{d-2}})\mid x-y\mid^{2-d}\nonumber\\
          & = & c_3\mid x-y\mid^{2-d}.
\end{eqnarray}
{ (c)} $\delta_B(y) > \widetilde r_0 $ and $L\mid x-y\mid  > \widetilde r_0.$ In this
case, we have $\delta_B(x)\geq\delta_B(y)\geq L\mid x-y\mid > \widetilde r_0.$ Choose a
point $w\in\partial B(y, \frac{\widetilde r_0}{2L}).$ Then from the argument in
{(b)} we get
$$G_B(w, y)\geq c_4\frac{1}{\left(\frac{\widetilde r_0}{2L} \right)^{d-2}}.$$
Now $B$ is connected, Lipschitz, and $\mid x-w\mid \leq \mid x-y\mid  +
\mid y-w\mid \leq\delta_B(y)/L + \frac{\widetilde r_0}{2L}.$ Therefore by the 
Harnack inequality (Proposition \ref{hq0}) and a chain argument, we have
$$G_B(x,y)\geq c_5G_B(w,y)\geq c_6\frac{1}{\left(\frac{\widetilde r_0}{2L}
  \right)^{d-2}}\geq c_72^{d-2}\frac{1}{\left(\frac{\widetilde r_0}{L}
  \right)^{d-2}}\geq c_8\mid x-y\mid^{2-d},$$
where in the last inequality, we used that $\mid x-y\mid  > \frac{\widetilde r_0}{L}.$

{\bf Case 2:}, $2\mid x-y\mid \leq\delta_B(y)< L\mid x-y\mid $

Take $x_0\in\partial B(y, \frac{\delta_B(y)}{L+1}).$ Then, 
$$\mid x-y\mid \leq\frac{1}{2}\delta_B(y)\leq L\mid x_0 -y\mid
=\frac{L}{L+1}\delta_B(y)\leq\delta_B(x)\wedge\delta_B(y).$$ We also
have $\mid x_0 - x\mid \leq\mid x_0 - y\mid + \mid x - y\mid
\leq\left(\frac{1}{L+1} + \frac{1}{2}\right)\delta_B(y).$ Therefore,
using Harnack inequality (Proposition \ref{hq0}) for the $\mathcal L$
harmonic function $G_B,$ and the argument in the Case $1$, we obtain
$$G_B(x, y)\geq c_9 G_B(x_0, y)\geq c_{10}\mid x_0-y\mid^{2-d}\geq c_{11}\mid x-y\mid^{2-d}.$$
 This finishes the proof of the lemma.
\end{proof}
Recall Notation \ref{not1}  where   $\phi$ was the smooth function describing the boundary $\partial B.$ For every $Q\in\partial B$ and $x\in B(Q, R_1)\cap B,$ we define $$\rho_Q(x) = x_d - \phi_Q(\tilde{x}),$$
where $(\tilde{x}, x_d)$ are the coordinates of $x$ in $CS_Q.$ 
\begin{notation} \label{not2} For a point
$x\in B,$ let $x^*$ be such that $\delta_B(x) = \mid x - x^*\mid .$ Let
$A_r(x^*)$ be such that $\delta_B(A_r(x^*))\geq \Lambda r$ and $\mid A_r(x^*) -
x^*\mid  = r.$ The constant $\Lambda$ depends only on the Lipschitz nature of
$B.$ For $0 < r < R_1 ,$ define
\begin{equation}\label{eqx}
\tilde x_r=\left\{\begin{array}{ll} x,& \delta_B(x)\geq\delta_B(A_r(x^*)),\\
A_r(x^*)& \delta_B(x) < \delta_B(A_r(x^*)).\end{array}\right.
\end{equation}
We remark that in the case of a ball, there is a canonical way to choose $A_r(x^*).$ 
\end{notation}

\begin{proposition}\label{3G}
Let $G_B$ denote the Green function for the
process $X_t$ killed on exiting $B.$ Then, there exists a positive constant $c_1$ such that
\begin{equation}\label{3gineq}
\frac{G_B(x, y) G_B(y, z)}{G_B(x, z)}\leq c_1\left[\frac{1}{\mid x - y\mid^{d-2}} + \frac{1}{ \mid y - z\mid^{d-2}}\right] \quad x, y, z\in B.
\end{equation}
\end{proposition}

\begin{proof}
We first prove the following claim.

{\bf Claim:} For $0 < r < R_1$ suppose $\mid x - y\mid  > 3r$ and $\mid x - z\mid  > 3r.$ If
(\ref{3gineq}) holds for $(\tilde x_r, y , z),$ then it holds for $(x,
y, z).$

\vspace{0.1in}

\noindent \textbf{Proof of Claim:}
If $\delta_B(x)\geq\delta_B(A_r(x^*)),$ $\tilde x_r = x,$ so the claim
holds trivially. So assume that $\delta_B(x) < \delta_B(A_r(x^*)).$ In
this case, $c_1 \leq \frac{\delta_B(\tilde x_r)}{r} \leq c_2$ and by
Theorem \ref{BdryHarnack} $$\frac{G_B(\tilde x_r, y)}{G_B(\tilde x_r,
  z)}\leq c_3 \frac{G_B(x, y)} {G_B(x,z)} \leq c_4 \frac{G_B(\tilde
  x_r, y)}{G_B(\tilde x_r, z)}. $$ However
$$\mid \tilde x_r-y\mid \geq \mid x-y\mid  - \mid x - \tilde x_r\mid \geq \mid x-y\mid  -r\geq\frac{2}{3}\mid x-y\mid .$$ This implies that (\ref{3gineq}) holds for $(x , y, z).$
\qed

Without loss of generality we may assume that $\delta_B(x)\leq\delta_B(z).$  Let $c_1 >0$ be such that $c_1\mid x-z\mid < R_1$ for all $x , z\in B.$ Set $c_2 = (\frac{1}{c_1} + 2).$

\noindent {\bf Case 1.} $\mid x-z\mid \leq \frac{c_2}{\Lambda}\delta_B(x).$ \\ If
$\mid x-z\mid \leq\delta_B(x)/2,$ then by (\ref{lower}) we have
$G_B( x, z)\geq c_3 \mid x-z\mid^{2-d}.$ Then using (\ref{upper}) we
 conclude that
$$\frac{G_B(x, y) G_B(y,z)}{G_B(x, z)}\leq c_4 \frac{\mid
   x-y\mid^{2-d}\mid y-z\mid^{2-d}}{\mid x-z\mid^{2-d}},$$ which
 implies (\ref{3gineq}). On the other hand, when $\mid x-z\mid
 >\delta_B(x)/2,$ select a point $\tilde z$ with $\mid x-\tilde z\mid =
 \delta_B(x)/2.$ Then by the Harnack inequality (Proposition \ref{hq0})
 $$ c_5 \leq \frac{G_B(x, z)}{G_B(x, \tilde z)}\leq \frac{1}{c_5} $$ and the lower bound in  
 (\ref{lower}) we have $G_B(x, \tilde z)\geq c_6\mid x-\tilde
 z\mid^{2-d}.$ Now $\delta_B(x)/2 < \mid x-z\mid < c_7\delta_B(x),$ and
 $\mid x-\tilde z\mid = \delta_B(x)/2.$ Thus $$\mid x-\tilde z\mid \leq 
 \mid x-z\mid  \leq c_8 \mid x-\tilde z\mid$$  and (\ref{3gineq}) holds.

 {\bf Case 2.} $\mid x-y\mid  < \frac{c_1 + 1}{\Lambda c_1}  \delta_B(x), \mid x-z\mid > \frac{c_2}{\Lambda} \delta_B(x).$ 

 In this case, $\mid z -x^*\mid \geq (\frac{c_2}{\Lambda} -
 1)\delta_B(x)$ and $\mid y -x^*\mid < (1+ \frac{c_1 + 1}{\Lambda
   c_1})\delta_B(x).$ Now when $\delta_B(y) < \delta_B(x),$ we use the
 Carleson estimate (\ref{Carleson}) to conclude $G_B(y , z)\leq c_9
 G_B(A_{\delta_B(x)}(y^*), z).$ However using a standard chain of balls argument
 and the Harnack inequality (Proposition \ref{hq0}) in $B\setminus{\{z\}}$
 from $A_{\delta_B(x)}(y^*)$ to $x$ (with length of chain independent of $\delta_B(x)$), we obtain
 $$c_{10} \leq \frac{G_B(x, z)}{ G_B(A_{\delta_B(x)}(y^*), z)} \leq  c_{11}.$$
 Thus $G_B(y, z)\leq c_{12}G_B(x, z)$ and with upper bound (\ref{upper})
 applied to $G_B(x , y)$ we obtain (\ref{3gineq}).

{\bf Case 3.} $\mid x-z\mid > \frac{c_2}{\Lambda}\delta_B(x),$ $\mid
y-z\mid > 2\mid x-z\mid , \mid x-y\mid > \frac{c_1 + 1}{\Lambda c_1}\delta_B(x).$

Set $r = c_1\mid x-z\mid $ so that $r < R_1.$ Then $\mid y-z\mid \geq
3r,$ $\mid x-z\mid \geq 3r$ and by earlier claim, it suffices to
consider $(x, y, \tilde z_r).$ Observing that $$\mid x-y\mid \geq \mid
y-z\mid - \mid x-z\mid \geq \mid x-z\mid \geq 3r$$ and
$$\mid x-\tilde z_r\mid \geq \mid x-z\mid  - r\geq (9-1)r\geq 3r$$
Again the claim proved earlier will imply that  it suffices to consider $(\tilde x_r, y, \tilde z_r).$ But 
$$ \mid \tilde x_r - \tilde z_r\mid \leq \mid x-z\mid  + 2r\leq(c_1^{-1} + 2)r\leq \frac{c_2}{c_1}(\delta_B(\tilde x_r)\wedge\delta_B(\tilde z_r)),$$
that is, we are back in {\bf Case $1$} and (\ref{3gineq}) holds.

{\bf Case 4.} $\mid x-z\mid  > \frac{c_2}{\Lambda}\delta_B(x),$ $\mid x-y\mid  > \frac{c_1 + 1}{\Lambda c_1} \delta_B(x), \mid y-z\mid  < 2\mid x-z\mid .$

Set $r= c_1\mid x-y\mid .$ Then observing that $\mid x-y\mid =
\frac{r}{c_1}$ and $$\mid x-z\mid \geq \mid x-y\mid - \mid y-z\mid
\geq \mid x-y\mid - 2\mid x-z\mid $$ which implies that $\mid x-z\mid
\geq\frac{1}{3}\mid x-y\mid \geq 3r,$ we see that Claim
applies again. Having switched to $(\tilde x_r, y, z),$ we note that
$$\mid \tilde x_r - y\mid  < \mid x-y\mid  + r = \frac{r}{c_1^{-1}} + r \leq
\frac{c_1 + 1}{\Lambda c_1}\delta_B(\tilde x_r).$$ If $\delta_B(\tilde x_r) <
\delta_B(z)$ we are in Case $1$ or Case $2$ and we are done. If
$\delta_B(\tilde x_r) > \delta_B(z)$ we would be done if $(z, y, \tilde
x_r)$ satisfies the conditions of either of the first three cases. So
we may assume the worst case scenario, that$(z, y, \tilde x_r)$ falls
into Case $4.$ However in that case, we may set $s = c_1\mid y-z\mid $ and use
the same argument as the first part of Case $4$ and it will follow
that $(\tilde z_s, y, \tilde x_r)$ or $(\tilde x_r, y, \tilde z_s)$
satisfies Case $1$ or $2$ and this completes the proof.
\end{proof}

We are now ready to prove  Proposition \ref{condgauge}. 

\vspace{0.1in}

\noindent \textbf{Proof of Proposition \ref{condgauge}:}
 Theorem 3.1 in \cite{chensong} contains a proof of  the proposition provided $q$ satisfies
 
\vspace{0.05in}

\noindent {\bf Assumption (Q1):}{\it 
There exists a Borel
set $K\subset B$ of finite measure and a $\delta >0$ such that
$$\beta(q):=\sup_{F\subset K: m(F)<\delta}\left(\sup_{(x,z)\in(B\times
  B)\setminus D}\int_{(B\setminus K)\cup
  F}\frac{G_B(x,y)G_B(y,z)}{G_B(x,z)}\mid q(y)\mid dy\right)< 1,$$
where $D$ denotes the diagonal $D = \{(x, x): x\in B\}$.  } 

\vspace{0.1in}

\noindent Using Proposition \ref{3G} there is a $c_3 >0$ such that
\begin{equation} \label{r3g}
\frac{G_B(x,y)G_B(y,z)}{G_B(x,z)} \leq c_3 \left ( \mid x-y \mid^{2-d} + \mid y-z \mid ^{2-d} \right )
\end{equation}
whenever $x,y,z \in B$. Therefore if $q: \mathbb{R}^d \rightarrow \mathbb{R}$ is such
that it satisfies (Q) then $q$ satisfies (Q1) thereby finishing the proof.
\qed

\vspace{0.1in}

\noindent We conclude the section with an estimate on the boundary behavior of
the Green Function. We shall prove that the Green function $G_B(x, y)$
decays atleast like $\delta_B(y)$ as $y$ approaches the boundary
$\partial B.$
\begin{proposition}\label{deltabd}
 Let $x$ be a fixed point in $B.$ Then, there exists a positive
 constant $c_1\equiv c_1(\alpha, R_1, M_1)$ such that for $r\in (0,R_1/2],$
   and for $Q\in\partial B,$
$$G_B(x, y)\geq c_1\frac{\delta_B(y)}{r},$$
for $y\in B\cap B(Q, rr_0/8).$ 
\end{proposition}
\begin{proof}
For any $r\in (0,R_1/2]$ and $y\in B\cap B(Q, rr_0/8),$ let $Q_y$ be
  the point so that $\mid y - Q_y\mid  = \delta_B(y)$ and let $y_0 = Q_y +
  \frac{r}{8}\frac{(y- Q_y)}{\mid y- Q_y\mid }.$ Choose a smooth function
  $\phi:\mathbb{R}^{d-1}\rightarrow\mathbb{R}$ satisfying $\phi(0) =
  \nabla\phi(0) = 0,$ and $\mid  \nabla\phi(x) - \nabla\phi(y)\leq
  M_1\mid x-y\mid $ and an orthonormal coordinate system $\mbox{CS} \equiv CS_Q$ with its
  origins at $Q_y$ so that $B(Q_y, R_1)\cap B = \{y = (y_d, \tilde
  y)\in B(0, R_1)\hspace{0.05in}\mbox{in CS}\hspace{0.05in}: y_d >
  \phi(\tilde y)\}.$

In the above coordinate system $\tilde y = 0,$ and $y_0 = (0, r/8).$
For $a_1, a_2 >0$ define
$$D(a_1, a_2) = \{y\hspace{0.05in}\mbox{in CS}\hspace{0.05in}: 0 < y_d
- \phi(\tilde y) < a_1\frac{r\delta_0}{8}, \mid \tilde y\mid <
a_2\frac{rr_0}{8}\}.$$ Then it is easy to see that $D(2, 2)\subset B\cap
B(Q, r/2).$ Since $G_B(x, \cdot)$ is a harmonic function in $B\cap B(Q,
r)$ and vanishes continuously in $B^c\cap B(Q, r),$ it is regular
harmonic in $B\cap B(Q, r/2)$ and hence also in $D(2, 2)$ 
( can be shown as in  Lemma $4.2$ \cite{cskv}). Thus by the Harnack inequality
\begin{eqnarray}\label{L25}
G_B(x, y)& = & \mathbb{E}^y\left(G_B(x, X_{\tau_{D(1,1)}}) \right) \nonumber\\
&\ge & \mathbb{E}^y\left(G_B(x, X_{\tau_{D(1,1)}}); X_{\tau_{D(1,1)}}\in D(2, 1) \right) \nonumber\\
&\ge & c_1 G_B(x, y_0)\mathbb{P}^y\left(X_{\tau_{D(1,1)}}\in D(2, 1) \right)\nonumber\\
&\ge & c_1 G_B(x, y_0)\frac{\delta_B(y)}{r},
\end{eqnarray}
where the last inequality is a consequence of Proposition \ref{bdryest2}. This concludes the proof.\end{proof}

\section{Martin Boundary, Density of Harmonic measure} \label{MBDH}
 Our aim in this section is  to prove Proposition \ref{conditional}. For this
 we need to understand the exit distribution of the process $X$. This
 requires us to introduce the notion of Martin boundary and show that
 it can be identified with the Euclidean boundary in the case of a
 ball. We also prove that the Martin Kernel gives the density of
 harmonic measure.

The notations and techniques are borrowed entirely from the literature
(\cite{Bass1}, \cite{cskv}, \cite{bb}). Fix $x_0\in B.$ Define $$ M(x, y) = \dfrac{G_B(x, y)}{G_B(x_0, y)}, x,
y\in B.$$ Note that $M(x_0, y) = 1$ for $y\in B,$ $y\neq x_0.$ 
\begin{notation}We
define the oscillation of a function $f: \mathbb{R}^d \rightarrow \mathbb{R}$ on a set
$A\subset \mathbb{R}^d$ by $$Osc_{A}f = \sup_{x \in A} f(x) - \inf_{x \in A} f(x).$$ Let $x \in
\bar B$. Define the box
$$ Q(x, a, b) = \{y\in B: \rho(y) < a, \mid (y_1, y_2, \ldots, y_{d-1}) - (x_1, x_2, \ldots, x_{d-1})\mid  < b\},$$
for $ a > 0, b >0$ and upper side of the box
 $$U(x,a ,b) = \{y\in\partial Q(x, a, b): \rho(y) = a\}. $$
\end{notation}
Our first result concerns the regularity of $M(x, \cdot)$ in a
neighborhood of the boundary $\partial B.$
\begin{proposition}\label{Hcont}
$M(x, y)$ is uniformly continuous in $y$ for $y$ in a neighborhood of $\partial B.$
\end{proposition}
\begin{proof}
 Pick $w\in\partial B$ and choose $r$ small enough
so that $B\cap B(w, r)$ is the intersection of $B(w, r)$ with the
region above the graph of a Lipschitz function. We note that $r$
depends on $B,$ but can be chosen independently of $w.$ Fix a
coordinate system as in Notation \ref{not1} and pick
$k_0$ large enough that $Q(w, 2^{-k_0}, 2^{-k_0})\subset B(w, r)$ and
$x, x_0\notin Q(w, 2^{-k_0}, 2^{-k_0}).$

Write $Q_k = Q(w, 2^{-k}, 2^{-k})$ and $h(y) = G_B(x_0, y).$ We will
show that the oscillation of $\frac{G_B(x, \cdot)}{h}$ on $Q_{k+1}$ is
controlled by the oscillation of $\frac{G_B(x, \cdot)}{h}$ on $Q_k.$ Both
$h$ and $G_B(x, \cdot)$ are harmonic functions on $B-\{x, x_0\}$ that
vanish on $\partial B.$ By the Harnack inequality (Proposition \ref{hq0})
they both are nonzero and finite for any $y\in Q_{k_0}.$ The Boundary
Harnack Principle (Theorem \ref{BdryHarnack}) in $Q_k,$ now gives that $\frac{G_B(x, \cdot)}{h}$ is bounded above and below by positive constants.

Fix $k\geq k_0.$ For $y \in B$, define $u(y) = \alpha G_B(x, y) + \beta
h(y),$ where $\alpha$ and $\beta$ are real numbers so that
$$ \sup_{Q_k}\left(\frac{u}{h} \right) = 1, \,\,
\inf_{Q_k}\left(\frac{u}{h} \right) = 0, \mbox{ implying } Osc_{Q_k}\left(\frac{u}{h}\right) = 1.$$
Clearly $u$ is harmonic in $Q_k.$ Now pick $z_k\in U(w, 2^{-k}, 2^{-k}).$ If
$\frac{u(z_k)}{h(z_k)}\leq 1/2,$ replace $u$ by $h-u.$ So we may
assume $\frac{u(z_k)}{h(z_k)}\geq 1/2$ without changing the supremum
and infimum of $u/h$ on $Q_k.$ By Theorem \ref{BdryHarnack}, if $y\in Q_{k+1}$
$$ \left(\frac{u}{h} \right)(y)\geq c_1\left(\frac{u}{h}
\right)(z_k)\geq c_1/2.$$ 
So $$Osc_{Q_{k+1}}\left(\frac{u}{h} \right)\leq 1-c_1/2 := \gamma$$
Undoing the algebra, we have
$$Osc_{Q_{k+1}}\left(\frac{G_B(x, \cdot)}{h(\cdot)} \right)\leq\gamma
Osc_{Q_{k}}\left(\frac{G_B(x, \cdot)}{h(\cdot)} \right)$$ and $\gamma <1.$ Now
$\frac{G_B(x, \cdot)}{h(\cdot)}$ is bounded by $c_2$ on $Q_{k_0}$ by Theorem \ref{BdryHarnack}. So $$ Osc_{Q_{k+1}}\left(\frac{G_B(x, \cdot)}{h(\cdot)} \right)\leq
c_2\gamma^{k-k_0},$$ or in other words,
$$\mid \frac{G_B(x, y)}{h(y)} - \frac{G_B(x, y')}{h(y')}\mid \leq c_3\gamma^{k}$$
if $y, y'\in Q(w, 2^{-k}, 2^{-k}).$ This implies the H\"{o}lder continuity of $\frac{G_B(x, y)}{h(y)}$ which in turn implies uniform continuity of $M(x \cdot)$.
\end{proof}
One consequence of Proposition \ref{Hcont} is that, if $y\to
z\in\partial B,$ then $\dfrac{G_B(x, y)}{G_B(x_0, y)} = M(x, y)$
converges.  We denote this limit by $M(x, z)$ and refer to it as the Martin Kernel. The next order of business is the so called Martin
Boundary. For the existence of the Martin Boundary, we refer the
reader to \cite{kunwat}. We summarize the result below as a lemma.
\begin{lemma}\label{MartinBdry}{[Kunita-Watanabe]}
There is a compactification $B_M$ of $B,$ unique upto homeomorphism, such that $M(x, y)$ has a continuous extension to $B\times (B_M\setminus\{x_0\})$ and $M(\cdot, z_1) = M(\cdot, z_2)$ iff $z_1 = z_2.$
\end{lemma}
The set $\partial B_M = B_M\setminus B$ is called the Martin Boundary
for $X^B.$ For $z\in\partial B_M,$ we set $M(\cdot , z) =0$ on $B^c.$ Now
$B_M$ is the smallest compact set for which $M(x, y)$ is continuous in
the extended sense in y. By Proposition \ref{Hcont}, $M(x, \cdot )$ is
uniformly continuous in a neighborhood of the boundary $\partial B,$
and so this implies that $B_M\subset\bar{B},$ and we can identify the
Martin boundary with a subset of the Euclidean boundary.
\begin{lemma}\label{identify}
Let $z_1, z_2\in\partial B.$ If  $M(\cdot , z_1)\equiv M(\cdot , z_2),$ then $z_1 = z_2.$
\end{lemma}
\begin{proof}
First we will show that $M(\cdot , z_1)$ vanishes on $\partial B -
\{z_1\}.$ Fix $y_0\in B,$ with $y_0\neq x_0.$ Let $w\in\partial B,
w\neq z_1.$ Note that since $G_B(x, y_0)\to 0$ as $x\to w,$ given
$\epsilon >0$ there is a $\delta < \mid w-z_1\mid /4$ such that $G_B(x,
y_0)\leq\epsilon$ if $\mid x-w\mid \leq\delta.$ Now applying Theorem \ref{BdryHarnack} in
$B\setminus(B(x, \delta/2)\cup B(x_0, \delta/2))$,
$$\frac{G_B(x, y)}{G_B(x_0, y)}\leq c_1\frac{G_B(x, y_0)}{G_B(x_0,
  y_0)}\leq\frac{c_1\epsilon}{G_B(x_0, y_0)},$$ if $\mid x-w\mid \leq\delta.$
Letting $y\to z_1,$ we see that $M(x, z_1)\to 0$ as $x\to w\neq z_1.$
Furthermore the convergence is uniform on compact subsets of $\partial
B - \{z_1\}.$

We next show that there is exactly one measure $\mu = \delta_{z_1}$
supported on $\partial B$ such that $M(x, z_1) = \int M(x, w)\mu(dw).$
If we show this, then $M(x, z_1) = M(x, z_2)$ will imply that
$\delta_{z_1} = \mu = \delta_{z_2},$ in other words, $z_1 = z_2.$

So suppose that $M(x, z_1) = \int_{\partial B}M(x, w)\mu(dw),$
$\mu\neq\delta_{z_1}.$ Note that $M(x_0, z_1) = 1 = \mu(\partial B).$
If $\mu\neq\delta_{z_1}$ there exists $\epsilon$ such that
$\mu_{\epsilon} = \mu\mid _{B(z_1, \epsilon)^c}\neq 0.$ Then $v(x) = \int
M(x, w)\mu_{\epsilon}(dw)$ is a harmonic function (see lemma
\eqref{integralHarmonic} below), bounded above by $M(x, z_1).$

In the first paragraph of this lemma, we showed that if
$\mid w-z_1\mid \geq\epsilon,$ then $M(x, w)\to 0$ as $x\to w'\in\partial
B\cap B(z_1, \epsilon/2).$ On the other hand, if $x\to w'\in\partial
B\cap B(z_1, \epsilon/2)^c,$ then

$$\int M(x, w)\mu_{\epsilon}(dw)\leq\int M(x, w)\mu(dw) = M(x, z_1)\to 0$$

 We have shown so far that $\lim_{x\to z}v(x) = 0$ for all
 $z\in\partial B.$ We also know from the definition of the Martin
 kernel that $v = 0$ in $B^c.$  We shall need the following Lemma

\begin{lemma}\label{vanishing}
Let $h:\mathbb{R}^d\rightarrow\mathbb{R}$ be a bounded $\mathcal L$
harmonic function in $B.$ Suppose that $h=0$ in $B^c$ and $\lim_{x\to
  z}h(x) = 0$ for all $z\in\partial B.$ Then, $h$ is identically $0.$
\end{lemma}

 Assuming this fact, let us continue with the argument. Applying Lemma \ref{vanishing}, we infer that $v(x) = \int M(x, w)\mu_{\epsilon}(dw) = 0.$ An application of Theorem \ref{BdryHarnack} lets us conclude
 that $\frac{G_B(x, y)}{G_B(x_0, y)}$ stays bounded below by a
 positive constant as $y\to \partial B.$ Therefore $M(x, w)$ is
 positive for all $w.$ But this must mean that $\mu_{\epsilon} = 0.$
 Since $\epsilon$ was arbitrary, $\mu(\{z_1\}^c) = 0.$ Recalling that
 $\mu(\partial B) = 1,$ we arrive at $\mu = \delta_{z_1}$ as was to be
 shown.
\end{proof}

\noindent \textbf{Proof of Lemma \ref{vanishing}:}
Let $D_n$ be an increasing sequence of open sets such that $\bar
D_n\subset D_{n+1}$ and $B = \cup_{n}D_n.$ Set $\tau_{D_n} = \tau_n$
for simplicity of notation. Note that $\tau_n\uparrow\tau_B.$ Now for
any $x\in B,$ using harmonicity of $h$ and the bounded convergence
theorem we have
\begin{eqnarray}
h(x)& = &\lim_{n\to\infty}\mathbb{E}^x\left(h(X_{\tau_n}); \tau_n <
\tau_B\right)\nonumber\\ & = &
\mathbb{E}^x\left(\lim_{n\to\infty}h(X_{\tau_n}); \tau_n < \tau_B\right)
\nonumber\\ & = & 0
\end{eqnarray}
which proves Lemma $12.$
\qed

\vspace{0.1in}

\noindent Combining the remark following Lemma \ref{MartinBdry}, and Lemma \ref{identify}, we obtain the following result
{\begin{theorem}\label{marteqbdry} There is a one-to-one mapping of
    the Martin boundary $\partial B_M$ onto the Euclidean boundary
    $\partial B,$ i.e the Martin boundary is the same as the Euclidean
    boundary.
\end{theorem}}
We next need the Martin representation for positive $\mathcal
L$-harmonic functions, for which we need to quote an abstract result
from the general theory of Markov processes.
\begin{lemma}\label{martinrep}
Every positive $\mathcal L$-harmonic function $h$ in the ball $B$ can be represented as
$$ h(x) = \int_{\partial B}M(x, y)\nu(dy),$$
for some positive measure $\nu$ concentrated on the boundary $\partial B.$
\end{lemma}
\begin{proof}
We first check that our process $X_t$ satisfies condition $(C)$ in
\cite{kunwat}. Then we apply Theorem $1$ in \cite{kunwat} in
combination with Theorem \ref{marteqbdry} above to obtain our required
result.
\end{proof}
The next three lemmas are proved in \cite{cskv}. For the reader's
convenience we present the proof here as well.
\begin{lemma}\label{lemmaCS1}
For every $z\in\partial B$ and every open set $U\subset\bar{U}\subset
B,$ $M(X_{\tau_U}, z)$ is $\mathbb{P}^x$ integrable.
\end{lemma}
\begin{proof} 
Take a sequence $z_n$ in $B\setminus{\bar U}$ such that $z_n\to z.$
Then using the fact that $M(\cdot , z_n)$ is regular harmonic in $U,$ and
Fatou's lemma, we have
\begin{eqnarray}
M(x, z)&=& \lim_{n}M(x, z_n) \nonumber\\
       &=& \lim_{n}\mathbb{E}^x\left(M(X_{\tau_U}, z_n)\right) \nonumber\\
       &\ge & \mathbb{E}^x\left(M(X_{\tau_U}, z)\right).
\end{eqnarray}
 Since $M(x, z)$ is finite, this concludes the proof of the lemma.
\end{proof}
\begin{lemma}\label{chensongG} 
For every $z\in\partial B,$ and $x\in B,$
$$ M(x, z) = \mathbb{E}^x\left(M(X_{\tau_{B(x, r)}}, z)\right),$$
for every $0 < r < \frac{1}{2}(R_1\wedge\delta_B(x)).$
\end{lemma}
\begin{proof}
Fix $z\in\partial B,$ $x\in B,$ and $r <
\frac{1}{2}(R_1\wedge\delta_B(x)).$ From the hypothesis on $r$ we have
that $B(z, r)\cap D\subset D\setminus B(x, r).$ Set $\eta_m =
(\frac{1}{2})^mr$ and let $z_m\in B $ be the point on the radial line
joining $z$ to the center $x_0$ of $B,$ such that $d(z_m, z) = \frac{3}{4}\eta_m.$
Then $z_m\in B(z, \eta_m)$. Let $\tau_r = \tau_{B(x, r)}$ be exit time
from $B(x, r).$ By the harmonicity of $M(\cdot , z_m),$ we have that
\begin{equation}\label{M}
M(x, z_m) = \mathbb{E}^x\left(M(X_{\tau_r}, z_m)\right)
\end{equation}
If we let $m\to\infty$ in \eqref{M}, we obtain $M(x, z)$ on the left
hand side. We will show that the random variables $M(X_{\tau_r}, z_m)$
are uniformly integrable, so that we can take the limit inside the
expectation in the right side of \eqref{M}.

As a preliminary, observe that from Theorem \ref{BdryHarnack} we know
there exists $c_1>0$ and $m_0\in\mathbb{N}$ such that for $m\geq m_0$
\begin{equation}
M(w, z_m) = \frac{G_B(w, z_m)}{G_B(x_0, z_m)}\leq c_1\frac{G_B(w, y)}{G_B(x_0, y)} = c_1M(w, y),
\end{equation}
for $w\in B\setminus B(z, \eta_m),$  $y\in B(z, \eta_{m+1}).$ Letting $y\to z\in\partial B,$ we obtain
\begin{equation}\label{M1}
M(w, z_m)\leq c_1M(w, z), \hspace{0.1in}m\geq m_0,
\end{equation}
for $w\in B\setminus B(z, \eta_m).$ Next, from Lemma \ref{lemmaCS1},
we know that $M(X_{\tau_r}, z)$ is $\mathbb{P}^x$
integrable. Therefore given $\epsilon > 0,$ there exists $N_0 > 1$
such that for $n\geq N_0,$
\begin{equation}\label{M2}
\mathbb{E}^x\left(M(X_{\tau_r}, z); M(X_{\tau_r}, z)> \frac{N_0}{c_1}\right)\leq\frac{\epsilon}{4c_1}.
\end{equation}
By (\ref{M1}) and (\ref{M2}),
\begin{eqnarray}
\lefteqn{\mathbb{E}^x\left(M(X_{\tau_r}, z_m); M(X_{\tau_r}, z_m) >
  N_0, \mbox{ and } X_{\tau_r}\in
  B\setminus B(z, \eta_m)\right)} \nonumber \\& \le &
c\mathbb{E}^x\left(M(X_{\tau_r}, z_m); c_1M(X_{\tau_r}, z)>
  N_0\right)\nonumber\\ & \le & c_1\frac{\epsilon}{4c_1} =
\frac{\epsilon}{4}
\end{eqnarray}
To deal with the other term, we use the Levy system formula \eqref{levy} and write 
\begin{eqnarray}\label{M3}
\lefteqn{\mathbb{E}^x\left(M(X_{\tau_r}, z_m); X_{\tau_r}\in B\cap B(z, \eta_m)\right)} \nonumber\\
&=& \int_{B\cap B(z, \eta_m)}M(w, z_m)\int_{B(x,r)}G_{B(x,r)}(x, y)J(y, w)dy dw \nonumber\\
& \leq & c_2\int_{B\cap B(z, \eta_m)}M(w, z_m)\int_{B(x,r)}\mid x-y\mid^{2-d}J(y,w)dy dw. \nonumber\\
\end{eqnarray}
Recall that $J(y, w)\leq c_3\mid y - w\mid^{-\alpha - d}.$ We have $\mid x-w\mid \geq \mid x-z\mid -\mid z-w\mid \geq 2r - \frac{r}{4} = \frac{7r}{4}.$ Ultimately we can estimate $\mid y-w\mid \geq \mid x-w\mid -\mid x-y\mid \geq \frac{7r}{4} - r = c_4r.$ Plugging this last information into \eqref{M3}, and using spherical coordinates to compute $\int_{B(x,r)}\mid x-y\mid^{2-d} = r^2,$ we obtain
\begin{equation}\label{M4}
\mathbb{E}^x\left(M(X_{\tau_r}, z_m); X_{\tau_r}\in B\cap B(z,
  \eta_m)\right)\leq c_5r^{2-\alpha -d}\int_{B(z, \eta_m)}M(w, z_m)dw
\end{equation}
Now using the definition of $M$ we rewrite the integral as 
\begin{equation}\label{Chenlast}
\int_{B(z, \eta_m)}M(w, z_m)dw = \frac{\int_{B(z, \eta_m)}G_B(w, z_m)dw}{G_B(x_0, z_m)}.
\end{equation}
Note that $B(z, \eta_m)\subset B(z_m, 2\eta_m).$ Therefore 
$\int_{B(z, \eta_m)}G_B(w, z_m)dw\leq\int_{B(z_m, 2\eta_m)} G_B(w, z_m)dw.$ Using the upper bound $G_B(w, z_m)\leq c_6\mid w-z_m\mid^{2-d}$ and converting the integral into polar coordinates, we obtain 
$$\int_{B(z, \eta_m)}G_B(w, z_m)dw\leq\int_{B(z_m, 2\eta_m)}G_B(w, z_m)dw\leq c_6\eta_m^2.$$ 
 On the other hand, from Proposition
\ref{deltabd}, we know that 
$$G_B(x_0, z_m)\geq c_7\delta_B(z_m)\geq c_8\eta_m.$$ Plugging the last two displays into equation \eqref{M4}, we infer that $$\mathbb{E}^x\left(M(X_{\tau_r}, z_m); X_{\tau_r}\in B\cap
  B(z, \eta_m)\right)\leq c_9\eta_m\to 0$$
as $m \rightarrow \infty$. Together with \eqref{M3}, this shows that the family
$M(X_{\tau_r}, z_m)$ is uniformly integrable and concludes the proof
of the lemma.
\end{proof}
\begin{lemma}\label{chenkim}
Let $z\in\partial B.$ Then, the function $x\rightarrow M(x, z)$ is harmonic in $B$ with respect to the process $X.$
\end{lemma}
\begin{proof}
Fix $z\in\partial B.$ Denote $h(x) = M(x, z).$ Let $U$ be an open set with $U\subset\bar U\subset B.$ For $x\in U,$ let us define 
$$ r(x) =
\frac{1}{2}(R_1\wedge\delta_B(x)),\hspace{0.1in}\mbox{and}\hspace{0.1in}
B(x) = B(x, r(x)). $$ Now define a sequence of stopping times
$\{T_n\}$ as follows.
$$T_1 = \inf\{t > 0: X_t\notin B(X_0)\},$$
and for $n\geq 2,$
\begin{equation}
T_n=\left\{\begin{array}{ll} T_{n-1} +
\tau_{B(X_{T_{n-1}})}\circ\theta_{T_{n-1}},& \mbox{if}\hspace{0.05in}
X_{T_{n-1}}\in U,\\ \tau_U& \mbox{otherwise} \end{array}\right.
\end{equation}
Observe that $X_{\tau_U}\in\partial U$ on the set $\bigcap_{n}\{T_n < \tau_U\}.$ Therefore, since $T_n\to\tau_U$ as $n\to\infty,$ $\mathbb{P}^x$ almost surely, and $h$ is continuous in $B,$ we have
$$ \lim_n h(X_{T_n}) = h(X_{\tau_U})\hspace{0.05in}\mbox{on}\hspace{0.05in} \bigcap_{n}\{T_n < \tau_U\}.$$
Next using that $h$ is bounded in $\bar U,$ and the dominated convergence theorem, we obtain
$$\lim_{n}\mathbb{E}^x\left(h(X_{T_n});\hspace{0.05in} \bigcap_{n}\{T_n <
  \tau_U\}\right) = \mathbb{E}^x\left(h(X_{\tau_U});\hspace{0.05in}
  \bigcap_{n}\{T_n < \tau_U\}\right).$$
Therefore, using Lemma \ref{chensongG},
\begin{eqnarray}
h(x)&=& \lim_{n}\mathbb{E}^x\left(h(X_{T_n})\right)\nonumber\\ &=&
\lim_{n}\mathbb{E}^x\left(h(X_{T_n}); \hspace{0.05in} \bigcup_{n}\{T_n =
  \tau_U\}\right) + \lim_{n}\mathbb{E}^x\left(h(X_{T_n});\hspace{0.05in}
  \bigcap_{n}\{T_n < \tau_U\}\right)\nonumber\\ &=&
\mathbb{E}^x\left(h(X_{\tau_U});\hspace{0.05in} \bigcup_{n}\{T_n =
  \tau_U\}\right) + \mathbb{E}^x\left(h(X_{\tau_U});\hspace{0.05in}
  \bigcap_{n}\{T_n < \tau_U\}\right) \nonumber\\ &=&
\mathbb{E}^x\left(h(X_{\tau_U})\right).
\end{eqnarray}
This shows that $h$ is harmonic and finishes the proof of the Lemma.
\end{proof}
 We will need another simple lemma that follows from
Lemma \ref{chenkim}.
\begin{lemma}\label{integralHarmonic}
Let $\mu$ be a finite measure supported in $\partial B.$ Then the function $h(x) = \int_{\partial B} M(x, w)\mu(dw)$ is harmonic in $B.$
\end{lemma}
\begin{proof}
Let $U\subset\bar{U}\subset B.$ We have to show that $h(x)
=\mathbb{E}^x(h(X_{\tau_U}))$ for $x\in U.$ To start with, we recall
that $M(\cdot , z)$ is defined to be $0$ in $B^c.$ Then for $x\in U,$ we
can write
\begin{eqnarray}
\mathbb{E}^x\left(h(X_{\tau_U})\right)&=& \mathbb{E}^x\left(h(X_{\tau_U});
  X_{\tau_U}\in B\right) + \mathbb{E}^x\left(h(X_{\tau_U}); X_{\tau_U}\notin B\right)
\nonumber\\ &=& \mathbb{E}^x\left(h(X_{\tau_U}); X_{\tau_U}\in B\right) + 0
\nonumber\\ &=& \mathbb{E}^x\left(\int_{\partial B} M(X_{\tau_U},
w)\mu(dw) \right) \nonumber\\ &=& \int_{\partial
  B}\mathbb{E}^x(M(X_{\tau_U}, w))\mu(dw)\nonumber\\ &=&
\int_{\partial B}M(x, w)\mu(dw)\nonumber\\ &=& h(x),
\end{eqnarray}
where we used lemma \ref{chenkim} to go from the fourth line to the fifth. This shows that $h$ is harmonic. 
\end{proof}
Next, we will introduce the notion of harmonic measure and derive the
density. Let $\omega^x(A) = \omega(x, A)= \mathbb{P}^x(X_{\tau_B}\in
A),$ where $A\subset\partial B$ is a Borel set. For fixed
$A\subset\partial B,$ $\omega (. , A)$ is a positive harmonic function
in $B$ vanishing on the complement of $\bar{B}.$ On the other hand,
for fixed $x,$ $\omega (x , \cdot)$ is a measure supported on $B^c.$ In general the harmonic measure will be supported on  $B^c,$ but we will restrict our attention to the measure on Borel subsets of the
boundary $\partial B.$

If $x, x_0 \in B,$ the Harnack inequality for $\mathcal L$
harmonic functions says $\omega^x$ is absolutely continuous with
respect to $\omega^{x_0}.$ Hence we know that a density exists,
$\omega^x(dy) = f(x, y)\omega^{x_0}(dy).$ The following theorem
identifies this density. We recall that $M$ denotes the Martin kernel
defined by
$$ M(x, y) = \dfrac{G_B(x, y)}{G_B(x_0, y)},$$
where $x_0$ is a fixed reference point and $x , y \in B.$ By the remark following Proposition \ref{Hcont}, we know that $\lim_{y\to z} M(x, y)$ exists for $z\in\partial B.$ We denote this limit by $M(x, z).$ Now we state our theorem.
\begin{theorem}\label{density}
Let $A\subset\partial B$ be a Borel set. Then, 
$$ \omega(x, A) = \int_{A} M(x, z)\omega(x_0, dz).$$
\end{theorem}
\begin{proof}
Fix $y_0\in\partial B.$ For each $k\in\mathbb{N},$ let $Q_{y_0k}$
denote the cube of the form $[\frac{j_1}{2^k}, \frac{j_1 +
    1}{2^k}]\times [\frac{j_2}{2^k}, \frac{j_2 + 1}{2^k}] ...\times
[\frac{j_d}{2^k}, \frac{j_d + 1}{2^k}]$ that contains $y_0$ and where
$j_1, \ldots,j_d$ are integers. Define
$$h_k (x) = \dfrac{\omega(x, Q_{y_0k}\cap\partial B)}{\omega(x_0, Q_{y_0k}\cap\partial B)}.$$
Note that $h_k$ is harmonic in $B$ and $h_k(x_0)=
1.$ From the Martin representation theorem, we know that there exists
a finite measure $\nu_k$ such that
$$h_k(x) = \int_{\partial B}M(x, z)\nu_k(dz).$$ Since $h_k(x_0) = 1$
and $M(x_0, z) = 1$ for all $z\in\partial B,$ we can infer that
$\nu_k(\partial B) = 1.$ Next, we know from the Lebesgue theorem for
radon measures, that the density $f$ can be expressed as
$$\lim_{k\to\infty} h_k(x) = \lim_{k\to\infty}  \dfrac{\omega(x, Q_{y_0k}\cap\partial B)}{\omega(x_0, Q_{y_0k}\cap\partial B)} = f(x, y_0).$$
We will now show that this limit is also $M(x, y_0).$ Indeed, first observe that since $\nu_k(\partial B) = 1,$  we can write
$$h_k(x) - M(x, y_0) = \int_{\partial B}M(x, z)\nu_k(dz) - \int_{\partial B}M(x, y_0)\nu_k(dz).$$
Therefore 
\begin{multline}\label{mainH}
\mid h_k(x) - M(x, y_0)\mid \leq \int_{\partial B\cap B(y_0, \delta_0)}\mid M(x,
z) - M(x, y_0)\mid \nu_k(dz)\\ + \int_{\partial B\cap B(y_0, \delta_0)^c}
\mid M(x, z) - M(x, y_0)\mid \nu_k(dz)
\end{multline}  
where $\delta_0 >0$ will be specified momentarily. Recall  that $M(x, \cdot)$ is uniformly continuous in $\partial B.$ Therefore given $\epsilon >0,$ there is a $\delta > 0,$ such that if $\mid z - y_0\mid  < \delta,$ $$\mid M(x, z) - M(x, y_0)\mid  < \epsilon.$$ Taking $\delta_0 = \delta$ in (\ref{mainH}), we obtain 
\begin{equation}\label{hmain2}
\mid h_k(x)- M(x, y_0)\mid \leq\epsilon + 2\mid \mid  M(x, \cdot)\mid \mid _{\infty}\nu_k(\partial
B\cap B(y_0, \delta)^c).
\end{equation}
We claim that for $k$ such that $Q_{y_0k}\subset B(y_0, \delta/2),$ $
\nu_k(\partial B\cap B(y_0, \delta)^c)= 0.$ Combined with equation
\eqref{hmain2} this will finish the proof of the theorem. To prove
this, take $k$ large as specified before and write $\mu_k =
\nu_k\mid _{\partial B\cap B(y_0, \delta)^c}.$ Let us define
$$u_k (x) = \int M(x, z)\mu_k(dz).$$
Then $u_k\leq h_k$ and $u_k$ is harmonic in $B$. 

We now recall the fact that if $\mid w-z_1\mid  > \eta,$ then $M(x, w)\to 0$ uniformly as $x\to w'\in B(z_1, \eta/2).$ Therefore using this, we have
$$\lim_{x\to w'}u_k (x) = 0$$ for $w'\in B(y_0, \delta/2)\cap\partial
B.$ Now let $w'\in B(y_0, \delta/2)^c\cap\partial B.$ We know that
$h_k$ vanishes continuously on $\partial B\setminus Q_{y_0(k-1)}.$
Therefore by the Carleson estimate (c.f Theorem \ref{Carleson}),
$\lim_{x\to w'} h_k(x) = 0,$ for $w'\in\partial B\setminus
Q_{y_0(k-1)}$. Since $u_k(x)\leq h_k(x)$ the same is true of $u_k$ as
well.

The upshot of all this is that $u_k$ tends to $0$ on the boundary
$\partial B.$ We already know that $u_k$ is $0$ on the complement of
$B.$ Since $u_k$ is harmonic, we have $u_k(x) = \mathbb
{E}^x(u_k(X_{\tau_B})).$ The boundary values are all $0,$ so $u_k(x) =
0$ in the ball $B.$ But $u_k(x) = \int M(x, z)\mu_k(dz)$ and $M(x, z)$
is positive for all $z\in\partial B$ (from Theorem \ref{BdryHarnack}). So this means
that $\mu_k = \nu_k\mid _{\partial B\cap B(y_0, \delta)^c} = 0.$
This finishes the proof of the theorem.
\end{proof} 
We are now ready to prove Proposition \ref{conditional}.

\vspace{0.1in}

\noindent \textbf{Proof of Proposition \ref{conditional}:}
By the Monotone class theorem, it is enough to prove the lemma for
$\phi$ of the form $\phi = 1_{\{t < \tau_B\}}\phi_t,$ where $\phi_t$
is a $\mathcal{F}_t$ measurable function. Next, we split the right
hand side as
 \begin{multline}\label{split}
\mathbb{E}^x\left(\mathbb{E}^x_{X_{\tau_B-}}(\phi); X_{\tau_B}\in A\right) = \mathbb{E}^x\left(\mathbb{E}^x_{X_{\tau_B-}}(\phi); X_{\tau_B}\in A, X_{\tau_B}=X_{\tau_B-} \right) \\ + \mathbb{E}^x\left(\mathbb{E}^x_{X_{\tau_B-}}(\phi); X_{\tau_B}\in A, X_{\tau_B}\neq X_{\tau_B-} \right)
\end{multline}
Let us label the two terms in equation (\ref{split}) as $I$ and $II$ respectively. We may write $I$ in the following way 
\begin{eqnarray}\label{split1}
I = \int_{A\cap\partial B}\mathbb{E}^x_z(\phi)\omega(x, dz) &=&
\int_{A\cap\partial B}\mathbb{E}^x_z(\phi)M(x, z)\omega(x_0,
dz)\nonumber\\ &=& \int_{A\cap\partial B}\mathbb{E}^x(t < \tau_B;
\phi_t M(X_t, z))\omega(x_0, dz) \nonumber\\ &=& \mathbb{E}^x\left(t <
\tau_B; \phi_t\int_{A\cap\partial B}M(X_t, z)\omega(x_0, dz)\right)
\nonumber\\ &=& \mathbb{E}^x\left(t < \tau_B; \phi_t\omega(X_t,
A\cap\partial B)\right) \nonumber\\ &=& \mathbb{E}^x\left(t < \tau_B;
\phi_t\mathbb{E}^{X_t}(X_{\tau_B}\in A\cap\partial B)\right)
\nonumber\\ &=& \mathbb{E}^x\left(1_{\{t < \tau_B\}}\phi_t; X_{\tau_B}\in
A\cap\partial B\right)\nonumber\\ &=& \mathbb{E}^x\left(\phi; X_{\tau_B}\in A;
X_{\tau_B}=X_{\tau_B-} \right).
\end{eqnarray}
\noindent To deal with $II,$ we first recall the Levy system formula,
equation (\ref{levy}) which gives us the joint density of
$(X_{\tau_B-}, X_{\tau_B})$ and so we can write
\begin{eqnarray}\label{split2}
II &=& \mathbb{E}^x\left(\mathbb{E}^x_{X_{\tau_B-}}(\phi); X_{\tau_B}\in A, X_{\tau_B}\neq X_{\tau_B-} \right) \\
&=& \int_{A}\int_{B}\mathbb{E}^x_z(\phi)G_B(x , z)J(z, y)dz dy \nonumber\\ 
&=& \int_{A}\int_{B}\mathbb{E}^x\{t < \tau_B; \phi_t G_B(X_t, z)\}J(z, y)dzdy \nonumber\\
&=& \mathbb{E}^x\left(t < \tau_B; \phi_t \int_{A}\int_{B}G_B(X_t, z)J(z, y) dz dy\right) \nonumber\\
&=& \mathbb{E}^x\left(t < \tau_B; \phi_t, \mathbb{E}^{X_t}(X_{\tau_B}\in A; X_{\tau_B}\neq X_{\tau_B-})\right) \nonumber\\
&=& \mathbb{E}^x\left(\phi; X_{\tau_B}\in A; X_{\tau_B}\neq X_{\tau_B-} \right).
\end{eqnarray}

Adding the two equations (\ref{split1}) and (\ref{split2}) gives the required result. The lemma is then proved. 
\qed

\label{fig:1}       
%
%

{\bf Acknowledgements:}
We thank Suresh Kumar for suggesting the
problem intiating the project; Richard Bass and Zhen-Qing Chen for useful discussions and references; and an anonymous referee for pointing out an error in a previous version of the article.



\end{document}